\documentclass[twoside,reqno]{amsart}

\usepackage{amsfonts,amsmath,amscd,amsthm,amssymb}
\usepackage{mathtools}
\usepackage[mathscr]{euscript} 
\usepackage{graphics}
\usepackage{wrapfig}
\usepackage{lscape}
\usepackage{rotating}
\usepackage{epsfig}
\usepackage{cite}
\usepackage{color}
\usepackage{booktabs}
\usepackage[vlined,ruled]{algorithm2e}
\usepackage{multirow}
\usepackage{todonotes}
\usepackage{ marvosym }

\usepackage{tikz}
\usetikzlibrary{patterns}
\newtheorem{example}{Example}[]

\usepackage[normalem]{ulem}

\newtheorem{lemma}{Lemma}
\newtheorem{theorem}{Theorem}
\newtheorem{remark}{Remark}
\newtheorem{definition}{Definition}

\DeclareMathOperator{\Div}{\operatorname{div}}

\DeclareMathOperator{\id}{\operatorname{id}}

\def\restrict#1{\raise-0.2ex\hbox{\ensuremath|}_{#1}}

\newcommand{\bld}[1]{\boldsymbol{#1}}

\newcommand{\compS}{\underline{\boldsymbol{\phi}}}
\newcommand{\compSh}{\underline{\boldsymbol{\phi}}_h}

\newcommand{\compU}{\underline{\boldsymbol{u}}}
\newcommand{\compUh}{\underline{\boldsymbol{u}}_h}

\newcommand{\compUhspace}{\underline{\boldsymbol{U}\!}_h}
\newcommand{\compUspace}{\underline{\boldsymbol{U}\!}(h)}

\newcommand{\compV}{\underline{\boldsymbol{v}}}
\newcommand{\compW}{\underline{\boldsymbol{w}}}

\newcommand{\compVh}{\underline{\boldsymbol{v}}_h}
\newcommand{\compWh}{\underline{\boldsymbol{w}}_h}

\newcommand{\Sh}{\bld{V}_{\!h}}

\newcommand{\pho}[1]{P_{\!h,0}^{#1}}
\newcommand{\ph}[1]{P_{\!h}^{#1}}
\newcommand{\Ph}[1]{\bld{P}_{\!h,0}^{#1}}

\newcommand{\curls}{{{ {\bld {\mathrm{curl}}}}}}

\newcommand{\grads}{{\nabla}}

\newcommand{\pol}{\mathbb{P}}

\newcommand{\bR}{\mathbb{R}}
\newcommand{\Oh}{\mathcal{T}_h}

\newcommand{\Eh}{\mathcal{F}_h}

\newcommand{\jmp}[1]{[\![#1 ]\!]}
\newcommand{\vertiii}[1]{{\left\vert\kern-0.25ex\left\vert\kern-0.25ex\left\vert #1
    \right\vert\kern-0.25ex\right\vert\kern-0.25ex\right\vert}}

\newcommand{\dx}{\,\mathrm{d}\bld x}
\newcommand{\ds}{\,\mathrm{d}\bld s}

\newcommand{\gradss}{\nabla_s}
\begin{document}
\title[$H(\Div)$-conforming HDG for linear elasticity]{
Locking free and gradient robust $H(\Div)$-conforming HDG 
methods for linear elasticity}
\author{Guosheng Fu}
\address{Department of Applied and Computational Mathematics and 
Statistics, University of Notre Dame, USA.}
\email{gfu@nd.edu}
\author{Christoph Lehrenfeld}
\address{
Institut f\"ur Numerische und Angewandte Mathematik, Lotzestr. 16-18, D-37083 G\"ottingen, Germany}
\email{lehrenfeld@math.uni-goettingen.de}
\author{Alexander Linke}
\address{Weierstrass Institute for Applied Analysis and Stochastics, Mohrenstr. 39, D-10117 Berlin, Germany}
\email{alexander.linke@wias-berlin.de}
\author{Timo Streckenbach}
\address{Weierstrass Institute for Applied Analysis and Stochastics, Mohrenstr. 39, D-10117 Berlin, Germany}
\email{timo.streckenbach@wias-berlin.de}


\begin{abstract}
  Robust discretization methods for
  (nearly-incompressible) linear elasticity are free of volume-locking and
  gradient-robust. While volume-locking is a well-known problem that can be
  dealt with in many different discretization approaches, the concept of
  gradient-robustness for linear elasticity is new. We discuss both aspects
  and propose novel Hybrid Discontinuous Galerkin (HDG) methods for linear elasticity. The starting point for these methods is a divergence-conforming discretization. As a consequence of its well-behaved Stokes limit the method is gradient-robust and free of volume-locking. To improve computational efficiency, we additionally consider discretizations with relaxed di\-ver\-gence-conformity and a modification which re-enables gradient-robustness,
yielding a robust and quasi-optimal discretization also in the sense of HDG superconvergence.
\end{abstract}
\maketitle

\section{Introduction}
\label{sec:intro}
Let $\Omega\subset \bR^d$, $d=2,3$, be a bounded polygonal/polyhedral domain. 
We consider the numerical solution of the 
isotropic linear elasticity problem
\begin{subequations}
\label{eqns}
\begin{alignat}{2} 
\label{eq1}
- \Div\left(2\mu\gradss \bld u\right)  -  \grads\left(\lambda\,\Div \bld u\right)  &= \bld f \qquad && \text{in $\Omega$,}\\
\label{eq2}
\bld u &= \bld 0 && \text{on $\partial \Omega$,}
\end{alignat}
\end{subequations}
where $\mu, \lambda$ are the (constant) Lam\'e parameters, $\bld u$ is the displacement,
$\gradss\bld u = {(\grads\bld u + 
\grads^T\bld u)}/{2}$ is the symmetric gradient operator, and $\bld f$ is an external force.
We consider homogeneous boundary condition for simplicity 
and focus on issues that are connected to
the fact that \eqref{eqns} is a vector-valued PDE,
and which arise in the nearly-incompressible limit
$\lambda\to\infty$. Indeed, the vector-valued displacements allow for
a natural, orthogonal splitting
\begin{equation}
  \bld u = \bld u^0 + \bld u^\perp
\end{equation}
in a divergence-free part $\bld u^0 \in \bld V^0$
and a perpendicular part $\bld u^\perp \in \bld V^\perp$
with
\begin{subequations}
\begin{align}
    \bld{V}^0 & := \left \{ \bld v \in \bld H^1_0(\Omega) : 
    \Div \, \bld v = 0  \right \}, \label{v-div}\\
    \bld{V}^\perp & := \left \{ \bld v \in \bld H^1_0(\Omega) : 
   \, (\gradss \bld v, \gradss \bld v^0)  = 0 
   \quad \text{for all $\bld v^0 \in \bld V^0$} \right \}
\end{align}
\end{subequations}
and the divergence-free part $\bld u^0$ can be easily shown to fulfill
the (formal) incompressible Stokes system
\begin{subequations} \label{eq:stokes}
\begin{alignat}{2}
  - \Div\left(2\mu\gradss \bld u^0\right) + \nabla p^0 & = \bld f
    \qquad && \text{in $\Omega$,} \\
      \Div\bld u^0 & = 0
    \qquad && \text{in $\Omega$,} \\
    \bld u^0 &= \bld 0 && \text{on $\partial \Omega$,}
\end{alignat}
\end{subequations}
where $p^0$ denotes a (formal) Stokes pressure, which serves
as the Lagrange multiplier for the divergence constraint
$\Div \bld u^0 = 0$.
Moreover, we will construct structure-preserving discretizations
for \eqref{eqns}, which allow for a reasonable discrete, orthogonal 
splitting
\begin{equation} \label{eq:discrete:splitting}
  \bld u_h = \bld u^0_h + \bld u^\perp_h,
\end{equation}
where also $\bld u^0_h$ is a discrete solution
of a discrete inf-sup stable and
{\em pressure-robust} space discretization of the incompressible
Stokes problem \eqref{eq:stokes} \cite{LinkeMerdon16}.
It is important to emphasize that the discrete splitting
\eqref{eq:discrete:splitting} is orthogonal, since numerical errors in
$\bld u^0_h$ cannot be compensated by contributions in
$\bld u^\perp_h$. 
Discrete inf-sup stability prevents --- which is well-known ---
the notorious Poisson (volume-) locking phenomenon, which
is a lack of optimal approximibility of divergence-free
vector fields by {\em discretely divergence-free} vector fields
\cite{BabuskaSuri1992}.
On the other hand,
{\em pressure-robustness} \cite{LinkeMerdon16,JLMNR2017}
for the Stokes part \eqref{eq:stokes} of problem \eqref{eqns}
avoids that gradient-fields in the force
balance incite numerical errors in
the displacements $\bld u$
due to an imperfect $\bld L^2$ orthogonality between gradient-fields
and {\em discretely divergence-free} vector fields.
For the incompressible Stokes problem \eqref{eq:stokes}, it was recently recognized
as similarly fundamental as inf-sup stability \cite{JLMNR2017,
MR4031577},
and it implies that only the divergence-free part of $\bld f$,
its so-called Helmholtz projector $\mathbb{P}(\bld f)$
\cite{JLMNR2017},
determines $\bld u^0$ --- and so $\bld u^0_h$ should be
determined by $\mathbb{P}(\bld f)$ only, as well.
In fact, recent investigations
show that pressure-robustness
becomes most important for multi-physics \cite{MR3481034} and
non-trivial high Reynolds number problems \cite{MR4031577}.
To put it simply, pressure-robustness guarantees
that a spatial discretization of the incompressible
Navier--Stokes equations in primitive variables possesses
an accurate, implicitly defined discrete vorticity equation
\cite{JLMNR2017}.

Similarly, the novel concept of
gradient-robustness for (nearly incompressible)
linear elasticity wants to assure
good accuracy properties of (an implicitly defined)
discrete vorticity equation for
the vorticity $\bld \omega := \curls ~ \bld u$.
The key idea to achieve this is that the
discrete $\bld L^2$ orthogonality between gradient-fields
and {\em discretely divergence-free} (test) vector fields
is the {\em weak equivalent of the vector calculus
identity $\curls ~ \nabla \psi  = \bld 0$
\cite{JLMNR2017}, which holds
for arbitrary smooth potentials $\psi$.}
We mention that this concept of {\em gradient-robustness}
can be introduced for quite a few vector PDEs. Recently,
it has already been introduced for the compressible barotropic
Stokes equations in primitive variables \cite{akbas2019gradientrobust}.

Concerning robustness of classical space discretizations
for nearly-incompressible linear elasticity, it is well-known that the classical low-order 
pure displacement-based conforming finite element methods suffer
from (Poisson) volume-locking, i.e., a
deterioration in performance in some cases as the 
material becomes incompressible. Various techniques have been introduced in
the literature to avoid volume-locking. This includes, for example,  the
high-order $p$-version
conforming methods \cite{Vogelius83,Scott85}, 
the technique of {\it reduced and selective integration}
\cite{Zienkiewics71,Hughes78} for low-order conforming methods, 
the nonconforming methods \cite{Falk91},
the discontinous Galerkin methods \cite{Hansbo02,Cockburn05},
various mixed methods
\cite{Arnold84,ArnoldDouglasGupta84,ArnoldWinther02,ArnoldWinther03,
  Arnold07,  GopalakrishnanGuzman11,PechsteinSchoberl11,
  GuzmanNeilan14},
the virtual element methods \cite{bdbd13},
the hybrid high-order methods \cite{dpdp15}, and 
the hybridizable discontinuous Galerkin (HDG) methods 
\cite{SoonCockburnStolarski09,CockburnShiHDGElas13,
QiuShi16,CockburnFu17c}.
However, none of the above cited references 
discusses about the property of {gradient-robustness}. 
It turns out that all of the above cited references, except the $p$-version
conforming methods \cite{Vogelius83,Scott85}
are not gradient-robust (see Definition \ref{def:gob} below).

Nevertheless, we conjecture
that gradient-robustness for (nearly-incompressible)
elasticity becomes important, whenever {\em strong and
complicated} forces of gradient type appear in the momentum balance.
In this contribution, we only want to discuss one possible
application coming from a {\em multi-physics} context,
i.e., we want to show how complicated gradient forces
may develop in elasticity problems:
In linear-thermoelastic solids the constitutive equation for the stress tensor reads as
\[
\sigma=\bld C\left\{ \varepsilon-\varepsilon^{\mathrm{th}}\right\}
\]
with $\varepsilon(\bld u) = \gradss \bld u$ and with
\[
\varepsilon^{\mathrm{th}}=\alpha(\theta-\theta_{0}) \bld I,
\]
where $\bld C$ and $\varepsilon$
denote the elasticity tensor and the linearized strain tensor.
Further, $\alpha$ denotes the (scalar) coefficient of linear
expansion and $\theta_{0}$ denotes a
(spatially and temporally) constant
reference temperature. For isotropic materials, this reduces to
\begin{align*}
  \sigma^{\mathrm{el}}&=\bld C \varepsilon=2\mu\varepsilon+\lambda\mathrm{tr}(\varepsilon) \bld I
  \\
  \sigma^{\mathrm{th}}&=\bld C \varepsilon^{\mathrm{th}}=(2\mu+3\lambda)\alpha(\theta-\theta_0 ) \bld I
 \\
  \sigma=\sigma^{\mathrm{el}} - \sigma^{\mathrm{th}} &= 2\mu\varepsilon+\lambda\mathrm{tr}(\varepsilon) \bld I-(2\mu+3\lambda)\alpha(\theta-\theta_0) \bld I
\end{align*}
with Lam\'{e} coefficients $\mu$, $\lambda$, see
\cite[pp. 528--529]{haupt}. Thus,
we finally obtain a momentum balance
\begin{equation}
- \Div\left(2\mu\gradss \bld u\right)
 - \grads\left(\lambda\,\Div \bld u\right)
  = - (2\mu+3\lambda)\alpha \Div \left ( \theta \bld I \right )
  = - (2\mu+3\lambda)\alpha \nabla \theta,
\end{equation}
where $-(2\mu+3\lambda)\alpha \theta$ denotes the potential
of a gradient force. For complicated and large temperature
profiles $\theta$ this gradient force can be made arbitrarily
complicated, in principle, and gradient-robustness should
be important in practice. However, in this contribution
we only want to study gradient-robustness from
the point of numerical analysis. Its (possible) importance in applications
will be investigated in subsequent contributions.

In this paper, we consider the discretization to \eqref{eqns} with 
divergence-conforming HDG methods
\cite{Lehrenfeld:10,LehrenfeldSchoberl16}, which are both
volume-locking-free and gradient-robust.

The rest of the paper is organized as follows:
In Section \ref{sec:mot}, we introduce the concepts of volume-locking and gradient-robustness by considering very basic discretization ideas for \eqref{eqns}.
Then, in Section \ref{sec2:disc}
we present and analyze
the divergence-conforming HDG scheme, in particular, we prove that the
scheme is both
locking-free and gradient-robust. 
In Section \ref{sec:relaxedhdiv} we 
consider and analyze two (more efficient) modified HDG schemes.
We conclude in Section \ref{sec:conclude}.

\section{Motivation: Volume-locking and gradient-robustness}
\label{sec:mot}
In this section we introduce the concepts of volume-locking and gradient-ro\-bust\-ness. To illustrate these we consider very basic discretization ideas
for \eqref{eqns} in this section and give a definition of volume-locking
and gradient-robustness. Only later, in the subsequent sections we turn our
attention to our proposed discretization, an $H(\Div)$-conforming HDG method and analyse it.

\subsection{A basic method}
Let us start with a very basic method. Let $\Oh=\{T\}$ be a conforming simplicial triangulation of $\Omega$. 
We use a standard vectorial $H^1$-conforming piecewise polynomial finite element space for the displacement function $\bld u$ in \eqref{eqns}:
$$
\Ph{k} : = [\pho{k}]^d \quad \text{with} \quad 
\ph{k} : = \prod_{T\in\Oh} \pol^{k}(T) \cap H^1(\Omega), \text{ and }
\pho{k}:=\ph{k}\cap H^1_0(\Omega)
$$
where $\pol^{k}(T)$ is the space of polynomials up to degree $k$. The numerical scheme is: Find $\bld u_h \in \Ph{k}$ s.t. for all $\bld v_h \in \Ph{k}$ there holds
\begin{equation} \tag{M1} \label{eq:naive}
a(\bld u_h, \bld v_h) \! := \!\!\int_{\Omega} 2\mu\,\gradss(\bld u_h)\!:\!\!\gradss(\bld v_h)  \dx
   + \int_{\Omega}  \lambda\,\Div(\bld u_h)\Div(\bld v_h)
   \dx
 = \int_{\Omega} \!\!\bld f \!\cdot\! \bld v_h \, \dx 
\end{equation}
We choose a simple numerical example to investigate the performance of the method.
\begin{example}\label{ex:1}
  We consider the domain $(0,1)^2$ and a uniform triangulation into right triangles. For the right hand side we choose
  $$
  \bld f = 2 \mu \pi^2 (\sin(\pi x)\sin(\pi y),\cos(\pi x)\cos(\pi y))
  $$
  and Dirichlet boundary conditions such that 
  $$\bld u = (\sin(\pi x)\sin(\pi y),
  \cos(\pi x)\cos(\pi y)
  )$$ is the unique solution.
\end{example}
\begin{figure}
  \includegraphics[height=0.255\textwidth]{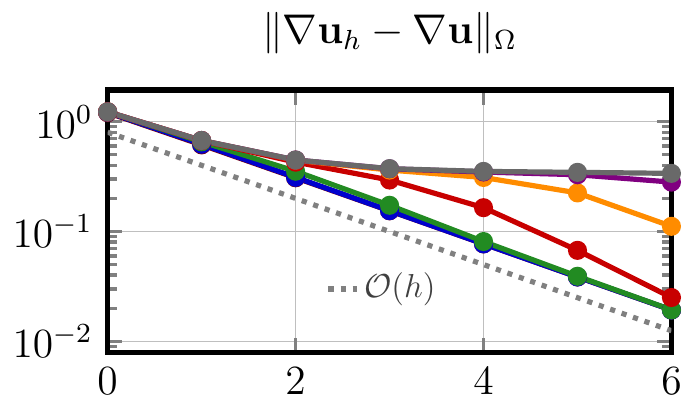}
  \hspace*{-0.02\textwidth}
  \includegraphics[height=0.255\textwidth]{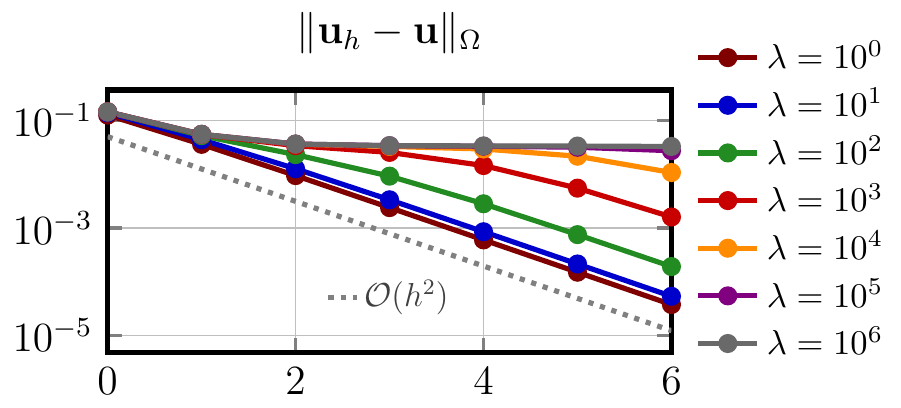}
  \vspace*{-0.06\textwidth}
  \caption{Discretization errors for the method \eqref{eq:naive}, $k=1$, under mesh refinement ($x$-axis: refinement level $L$) and different values of $\lambda$ for Example \ref{ex:1}.}
  \vspace*{-0.02\textwidth}
  \label{fig:M1startP1}
\end{figure}

For successively refined meshes with smallest edge length $h=2^{-(L+2)}$, fixed polynomial degree $k=1$ and levels $L=0,..,6$ we compute the error $\bld u - \bld u_h$ in the $L^2$ norm and the $H^1$ semi-norm for different values of $\lambda$. The absolute errors are displayed in Figure \ref{fig:M1startP1}. Let us emphasize that the solution $\bld u$ is independent of $\lambda$.
 For fixed and moderate $\lambda$ we observed the expected convergence rates, i.e. second order in the $L^2$ norm and first order in the $H^1$ norm. However, we observe that the error is severely depending on $\lambda$. Especially for larger values of $\lambda$ the asymptotic convergence rates for $h \to 0$ are shifted to finer resolutions; for instance, for $\lambda=10^5$ convergence can not yet be observed on the chosen meshes. Overall, we observe an error behavior of the form $\mathcal{O}(\lambda \cdot h^{k})$ for the $H^1$ semi-norm and $\mathcal{O}(\lambda \cdot h^{k+1})$ for the $L^2$ norm.
From the discretization \eqref{eq:naive} we directly see that with
increasing $\lambda$ we enforce that $\Div\bld  u$ tends to zero (pointwise). For piecewise linear functions, however, the only divergence-free function that can be represented is the constant function.
This leads to the observed effect which is known as \emph{volume-locking}. We give a brief definition here:

Volume-locking is a structural property of the discrete finite element
spaces involved. In the limit case $\lambda \to \infty$, one
expects that the limit displacement $\bld u$ is divergence-free.
Recalling \eqref{v-div} and introducing
the discrete counterpart
\begin{equation}
  \bld{V}^0_h := \left \{ \bld v_h \in \Ph{k} : 
    \Div_h \, \bld v_h = 0  \right \},
\end{equation}
where $\Div_h$ is a discretized $\Div$ operator,
one is ready for a precise definition of volume-locking:
\begin{definition} \label{def:vollock}
Volume-locking means that the discrete subspace of discretely divergence-free
vector fields of $\bld V^0_h$ does not have optimal approximation properties
versus smooth, divergence-free vector fields
$\bld v \in \bld V^0 \cap \bld H^{k+1}(\Omega)$
\begin{subequations}
\begin{equation}
  \inf_{\bld v_h \in \bld V^0_h}
   \|\nabla \bld v- \nabla \bld v_h\|_{\Omega}
\not\leq C h^k | \bld v  |_{k+1},
\end{equation}
although the entire vector-valued finite element space
$\Ph{k}$ possesses optimal approximation properties of the form
\begin{equation}
  \inf_{\bld v_h \in \Ph{k}}
   \|\nabla \bld v- \nabla \bld v_h\|_{\Omega}
\leq C h^k | \bld v  |_{k+1}.
\end{equation}
\end{subequations}
\end{definition}

\begin{figure}[t]
  \includegraphics[height=0.255\textwidth]{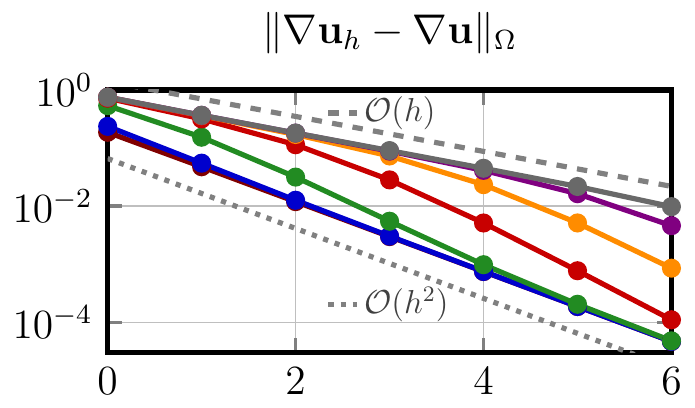}
  \hspace*{-0.02\textwidth}
  \includegraphics[height=0.255\textwidth]{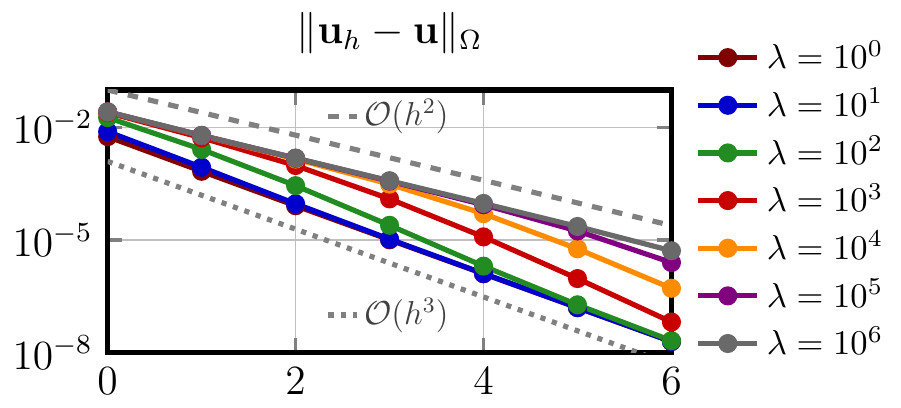}
  \hspace*{-0.06\textwidth}
  \caption{Discretization errors for the method \eqref{eq:naive}, $k=2$, under mesh refinement ($x$-axis: refinement level $L$) and different values of $\lambda$ for Example \ref{ex:1}.}
  \hspace*{-0.02\textwidth}
  \label{fig:M1startP2}
\end{figure}

In the sense of Definition \ref{def:vollock} the discretization \eqref{eq:naive} with $k=1$ is obviously not free of volume-locking.
The problem can be alleviated by going to higher order, cf. Figure
\ref{fig:M1startP2} for the same problem and discretization but with order $k=2$. 
We observe that convergence is secured in this case also for the highest values of $\lambda$. However, the discretization error still depends strongly on $\lambda$ and for large $\lambda$ and insufficiently fine mesh sizes $h$ an order drop can be observed. The overall convergence behaves 
like $\mathcal{O}(\min\{ h^{k-1}, \lambda h^{k} \})$ for the $H^1$
semi-norm and $\mathcal{O}(\min\{ h^{k}, \lambda h^{k+1} \})$ for the $L^2$
norm. Hence, even for $k=2$ the discretization \eqref{eq:naive} is not free of volume-locking.

\subsection{A volume-locking-free discretization through mixed formulation}
To get rid of the locking-effect one often reformulates the grad-div term in \eqref{eqns} by rewriting the problem in mixed form as 
\begin{subequations}
  \label{mstokes}
  \begin{alignat}{2} 
  \label{eq:mstokes:1}
  - \Div\left(2\mu\gradss \bld u\right)  -  \nabla p  &= \bld f \qquad && \text{in $\Omega$,}\\
  \label{eq:mstokes:2}
  \Div \bld u + \lambda^{-1} p & = 0 \qquad && \text{in $\Omega$,}
  \end{alignat}
\end{subequations}
Here, the auxiliary variable $p$ approximating $\lambda \Div \bld u$ is introduced.
In the limit $\lambda \to \infty$ this yields an incompressible Stokes problem. With the
intention to avoid volume-locking we now consider a discretization that is
known to be stable in the Stokes limit. Here, we take the well-known Taylor-Hood velocity-pressure pair: Find $(\bld u_h, p_h) \in \Ph{k}\times\ph{k-1}$, s.t. 
\begin{subequations}
  \label{eq:mdstokes}  
  \begin{alignat}{2} 
  \label{eq:mdstokes:1}  \tag{M2a}
  \int_{\Omega} 2\mu\,\gradss(\bld u_h):\gradss(\bld v_h)  \dx + \int_{\Omega}  \Div(\bld v_h) p_h  \dx & = \int_{\Omega} \bld f \!\cdot\! \bld v \, \dx   \quad&& \forall~v_h \in \Ph{k},\\
\int_{\Omega} \Div(\bld u_h) q_h  \dx -
\int_{\Omega} \lambda^{-1} p_h q_h \dx 
& = 0 \quad&& \forall~q_h \in \ph{k-1}.
  \label{eq:mdstokes:2}  \tag{M2b}
  \end{alignat}
\end{subequations}
It is well-known that for every LBB-stable Stokes discretization the mixed formulation of linear elasticity guarantees that the discretization is free of volume-locking in the sense of Definition \ref{def:vollock},
cf. \cite[Chapter VI.3]{braess2013finite}.

Let us note that we can interprete \eqref{eq:mdstokes:2} as $p_h = \lambda \Pi_{\ph{}} \Div(\bld u_h)$ where $\Pi_{\ph{}}$ is the $L^2(\Omega)$ projection into $\ph{k-1}$. Hence, we can formally rewrite (M2) as: Find $\bld u_h \in \Ph{k}$ s.t.
\begin{equation} \tag{M2*} \label{eq:th2}
  \int_{\Omega} 2\mu\,\gradss(\bld u_h):\gradss(\bld v_h)  \dx + \int_{\Omega}  \lambda\,\Pi_{\ph{}} \Div(\bld u_h)\Div(\bld v_h)  \dx = \int_{\Omega} \bld f \cdot \bld v \, \dx. 
\end{equation}
We note that the only difference between \eqref{eq:th2} and \eqref{eq:naive} is in the projection $\Pi_{\ph{}}$.
Hence, the $\Div_h$ in a corresponding subspace $\bld{V}^0_h$ is different, 
$$  \bld{V}^0_h := \left \{ \bld v_h \in \Ph{k} : 
  \Pi_{\ph{}} \Div \left(  \bld v_h \right) = 0  
\right \},$$
yielding a much richer space $\bld{V}^0_h$ to approximate with.
The scheme \eqref{eq:th2} can be considered as an improvement over 
 the plain scheme \eqref{eq:naive} using a
 {\it reduced integration} \cite{Hughes78} 
  for the grad-div term to avoid volume-locking.
  See \cite{Malkus78} 
  for a discussion of the
  equivalence of certain mixed finite element methods with displacement
  methods which use the {\it reduced and selective integration} technique.

\begin{figure}
  \includegraphics[height=0.255\textwidth]{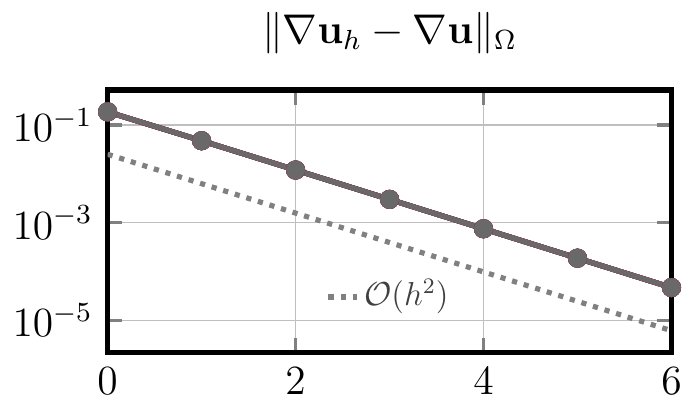}
  \hspace*{-0.02\textwidth}
  \includegraphics[height=0.255\textwidth]{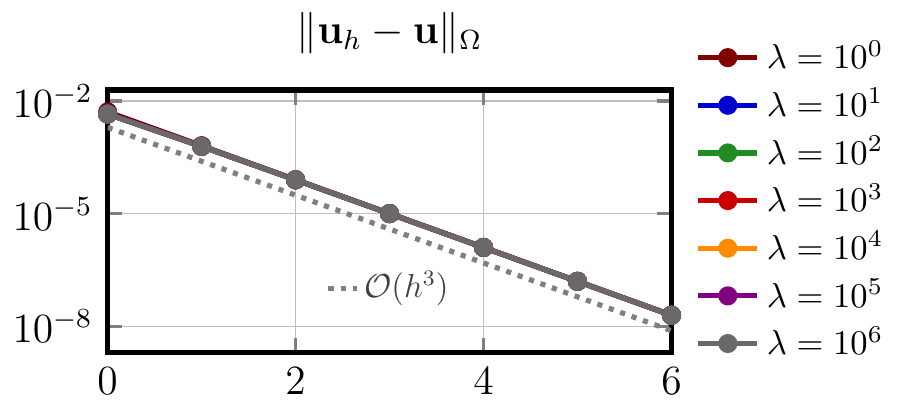}
  \hspace*{-0.06\textwidth}
  \caption{Discretization errors for the method \eqref{eq:mdstokes}, $k=2$, under mesh refinement ($x$-axis: refinement level $L$) and different values of $\lambda$ for Example \ref{ex:1}.}
  \label{fig:M2startP2}
  \hspace*{-0.02\textwidth}
\end{figure}

In Figure \ref{fig:M2startP2} we display the results of the previous numerical experiment with the method in \eqref{eq:mdstokes:1}--\eqref{eq:mdstokes:2}. We observe that indeed, the discretization error is essentially independent of $\lambda$ and optimally convergent. 


\subsection{Gradient-robustness}
In the previous subsection we considered a divergence-free force field. As a result of the Helmholtz decomposition we can decompose every $L^2$ force field into a divergence-free and an irrotational part. In this section we now consider the case where the force field is irrotational, i.e. a gradient of an $H^1$ function. This will lead us to the formulation of \emph{gradient-robustness}. 
Assume that there is $\phi \in H^1(\Omega)$
with $\int \phi \, \mathit{dx} = 0$ so that $\bld f = \nabla \phi$. With $\lambda \to \infty$ we have $p \to \phi$ and $\bld u \to \bld 0$
, i.e. in the Stokes limit gradients in the force field are solely balanced by the pressure and have no impact on the displacement.
In the next subsection,
this reasoning will be made more precise by
deriving an asymptotic result in the limit $\lambda \to \infty$.

\subsection{A definition of gradient-robustness} \label{ssec:defgradrob}
First, we introduce the
orthogonal complement 
of the weakly-differential divergence-free vector fields \eqref{v-div}
with respect to the inner-product
$a(\cdot, \cdot)$ defined in \eqref{eq:naive}:
\begin{equation}
    \bld{V}^\perp := \{ \bld u \in \bld{H}^1_0(\Omega) :
   a(\bld u, \bld v) = 0,
      \forall \bld v \in \bld{V}^0 \}.
\end{equation}
Then, the solution of the linear elasticity equation can be
decomposed as
\begin{equation}
  \bld u = \bld u^0 + \bld u^\perp, \qquad \bld u^0 \in \bld V^0,
    \quad \bld u^\perp \in \bld V^\perp,
\end{equation}
where $\bld u^0$ satisfies
 \begin{equation}
   \label{eq-decomp}
   a(\bld u^0,\bld v^0) = a(\bld u, \bld v^0)
   =(\bld f, \bld v^0),\quad \forall \bld v^0\in
   \bld V^0.
 \end{equation}

The following lemma characterizes a robustness property of exact solutions to linear elasticity problems.
\begin{theorem}[Gradient-robustness of nearly incompressible materials] \label{lem:gradient:field:robustness}
If the right hand side $\bld f \in H^{-1}(\Omega)$ in \eqref{eq1} is a gradient field, i.e. $\bld f=\nabla \phi$, $\phi \in L^2(\Omega)$, then it holds for the solution  $\bld u = \bld u^0 + \bld u^\perp$ of \eqref{eqns} (under homogeneous Dirichlet boundary conditions)
$$
  \bld u^0 = 0, \qquad \bld u^\perp = \mathcal{O}(\lambda^{-1}),
$$
i.e., for $\lambda \to \infty$ one gets
$\bld u = \bld u^\perp \to \bld 0$. 
\end{theorem}
\begin{proof}
  Taking $\bld v^0=\bld u^0$ in equation \eqref{eq-decomp}, we get
    \[
      a(\bld u^0,\bld u^0) = (\bld f,\bld u^0)
      =(\nabla \phi, \bld u^0)  =
      (-\phi, \Div(\bld u^0)) = 0.
    \] 
    Hence, $\bld u^0=0$.
    
  On the other hand we obtain
  \begin{align*}
    ( 2\mu\,\gradss(\bld u^\perp), \gradss(\bld u^\perp))
    + ( \lambda\,\Div \bld u^\perp, \Div \bld
    u^\perp) &= \bld f(\bld u^\perp) \\
    = -(\phi,\Div \bld u^\perp) &\leq \Vert \phi \Vert_{L^2(\Omega)} \Vert \bld u^\perp \Vert_{H^1(\Omega)}. 
  \end{align*}
  From Korn's inequality
  $
  \Vert \bld u^\perp \Vert_{H^1(\Omega)}^2 
 \le C
  ( 2 \gradss(\bld u^\perp), \gradss(\bld u^\perp)),
  $
  and an estimate on the $H^1$ norm of functions in $\bld V^\perp$, 
   $\Vert \bld u^\perp \Vert_{H^1(\Omega)} \leq \beta \Vert \Div \bld u^\perp \Vert_{L^2(\Omega)}$, where 
$C$ is the constant for the Korn's
   inequality and $\beta$ is the inf-sup constant of a corresponding Stokes problem, cf. \cite[Corollary 3.47]{john2016finite},
   we hence have 
   $$ 
(\frac{\mu}{C}+\frac{\lambda}{\beta}) \Vert \bld u^\perp
     \Vert_{H^1(\Omega)} \le \Vert \phi \Vert_{L^2(\Omega)},
   $$
   from which we conclude the statement.
\end{proof}
The previous characterization does not automatically carry over to discretization schemes.

\begin{definition}
  \label{def:gob}
  We denote a space discretization for the linear elasticity equation
  which fulfills an analogue to Theorem \ref{lem:gradient:field:robustness} also discretely as \emph{gradient-robust}, i.e., gradient-robustness means for
  a discretization  of \eqref{eqns} that in the limit $\lambda \to \infty$
  it holds $\bld u_h = \mathcal{O}(\lambda^{-1})$.
\end{definition}

\begin{remark}[Gradient robustness for the Stokes limit]
  \emph{Gradient-robustness} is directly related to the concept of \emph{pressure robustness} in the Stokes case.
  Actually, a gradient-robust discretization
  for the linear elasticity problem \eqref{eqns} is 
  asymptotic preserving (AP) in the sense of 
  \cite{MR1718639} such that for $\lambda \to \infty$
  the space discretization converges on every (fixed) grid
  to a pressure-robust space discretization of the
  (formal) Stokes problem
  \eqref{eq:stokes}.
\end{remark}
  It is known that the standard Taylor-Hood discretization is not {pressure-robust}. However, several discretizations  for the Stokes problem exists that are {pressure-robust} \cite{JLMNR2017} or can be made {pressure-robust} by a suitable modification \cite{Link14}.
We demonstrate the consequences for the linear elasticity problem in the following, 
where the forcing $\bld f$ is a gradient field.
\begin{example}\label{ex:2}
We take $\bld f = \nabla \phi$ with $\phi = x^6 + y^6$.
and (homogeneous) Dirichlet boundary conditions so that
it holds $\bld u \to \bld 0$ in the asymptotic limit $\lambda \to \infty$.
\end{example}
We now compare the different methods on a fixed grid
(or a couple of fixed grids) and we investigate
the norm of the solution $\bld u$ with respect to $\lambda \to \infty$.
For gradient-robust methods this norm should vanish as $\mathcal{O}(\lambda^{-1})$ independent of $h$.
For methods that are not gradient-robust the limit 
will be $\mathcal{O}(h^k)$ for $\lambda \to \infty$
depending on the mesh size $h$ and the order $k$.
The results for the methods \eqref{eq:naive} and \eqref{eq:mdstokes:1}--\eqref{eq:mdstokes:2} for Example \ref{ex:2}, are shown in Fig. \ref{fig:M2grad}. While \eqref{eq:naive} behaves well as $\Vert \nabla \bld u_h \Vert_{\Omega}$ goes to zero with $\lambda^{-1}$ essentially independent of $h$, for the method in  \eqref{eq:mdstokes:1}--\eqref{eq:mdstokes:2} we observe a lower bound for $\Vert \nabla \bld u_h \Vert_{\Omega}$ that depends on the mesh.

\begin{figure}
 \includegraphics[height=0.255\textwidth]{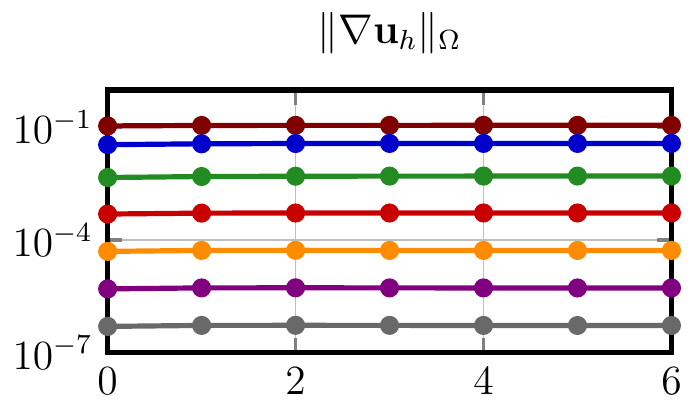}
 \hspace*{-0.02\textwidth}
 \includegraphics[height=0.255\textwidth]{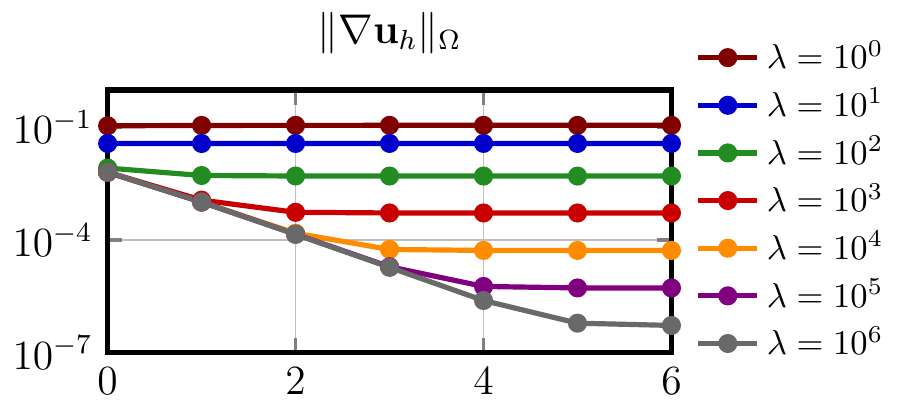}
 \caption{Norm of discrete solution for methods in \eqref{eq:naive} (left) and \eqref{eq:mdstokes:1}--\eqref{eq:mdstokes:2} (right), $k=2$, under mesh refinement ($x$-axis: refinement level $L$) and different values of $\lambda$ for Example \ref{ex:2}.} 
 \label{fig:M2grad}
 \hspace*{-0.02\textwidth}
\end{figure}

As a conclusion of the numerical examples, let us summarize that both basic
methods that we considered here, the discretization \eqref{eq:naive} and
the Taylor-Hood based method in
\eqref{eq:mdstokes:1}--\eqref{eq:mdstokes:2} are not satisfactory. While
\eqref{eq:naive} seems to be \emph{gradient-robust} it is not free of volumetric locking while the behavior of the Taylor-Hood based method in \eqref{eq:mdstokes:1}--\eqref{eq:mdstokes:2} has the exact opposite properties.

\begin{remark}
  In the first example we considered a divergence-free forcing. The observations stay essentially the same if more general forcings are considered there, e.g. if the solutions of both examples are superimposed.
\end{remark}

\subsection{The basic method on 
barycentric refined meshes
}
A comparably simple discretization scheme that is known to be pressure
robust for the Stokes limit is the Scott-Vogelius element \cite{Scott85,Vogelius83}, which is the
classical Taylor-Hood discretization with a discontinuous pressure space.
However, this discretization is known to be LBB-stable (and hence free of volume-locking) only for special triangulations or sufficiently high orders. Applications of this element
to linear elasticity have been made for example in \cite{MR2871869}. Let us consider the last two examples again, but on every level we apply a barycentrical refinement 
of the original mesh by connecting all vertices of the mesh cell with
the barycenter of this mesh cell, cf. Figure \ref{fig:bary}.
\begin{figure}
  \begin{center}
    \begin{tikzpicture}[scale = 0.7]
        \def\dw{0.025}
        \foreach \i in {0,...,3} {
          \foreach \j in {0,...,3} {
            \draw [] (\i,\j) -- (\i,1+\j);
            \draw [] (\i,\j) -- (\i+1,\j);
            \draw [] (\i,\j) -- (\i+0.33333,\j+0.33333);
            \draw [] (\i+0.33333,\j+0.33333) -- (\i+1,\j);
            \draw [] (\i+0.33333,\j+0.33333) -- (\i,\j+1);
            \draw [] (\i+1,\j) -- (\i,\j+1);
            \draw [] (\i+0.66666,\j+0.66666) -- (\i,\j+1);
            \draw [] (\i+0.66666,\j+0.66666) -- (\i+1,\j+1);
            \draw [] (\i+0.66666,\j+0.66666) -- (\i+1,\j);
          }
        }
        \draw [] (4,0) -- (4,4) -- (0,4);
    \end{tikzpicture}
    \vspace*{-0.3cm}
  \end{center}    
  \caption{Barycentric-refined triangular mesh on the unit square with refinement level $L=0$.} 
  \label{fig:bary}
  \hspace*{-0.02\textwidth}
\end{figure}
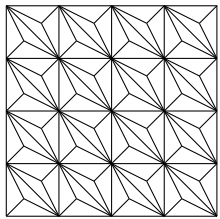
If we apply the basic method \eqref{eq:naive} in this case with $k\geq 2$ we have a gradient-robust scheme which at the same time has a sufficiently large discretely divergence-free subspace $\bld V_h^0$ to be volume-locking free. The results are given in Figure \ref{fig:PS} and are consistent with these expectations.

\begin{figure}
  \begin{center}
    \includegraphics[height=0.255\textwidth]{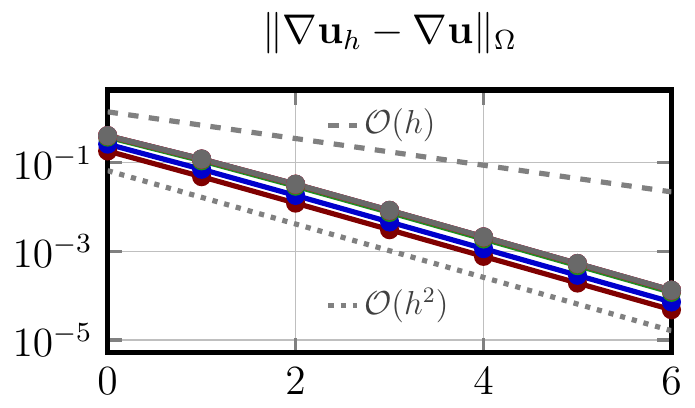}
    \hspace*{-0.02\textwidth}
    \includegraphics[height=0.255\textwidth]{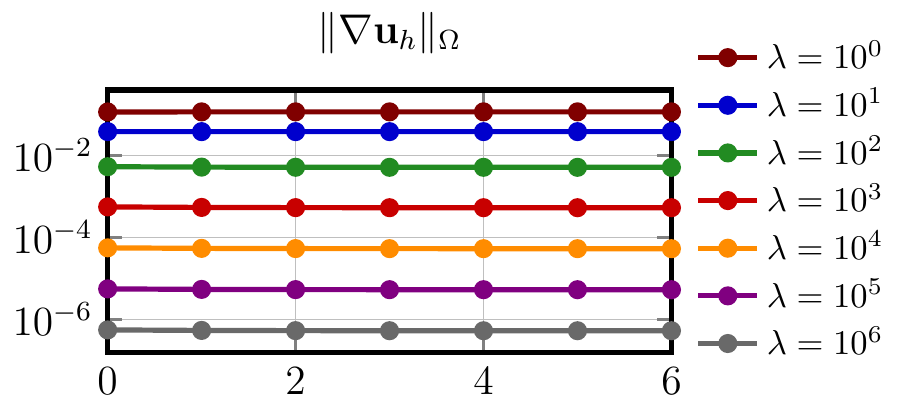}
  \end{center}    
  \caption{Discretizaton error for Example \ref{ex:1} (left) and norm of discrete error for Example \ref{ex:2} (right) for the method \eqref{eq:naive}, $k=2$, on a barycentric-refined mesh under mesh refinement ($x$-axis: refinement level $L$) and different values of $\lambda$ for Example \ref{ex:2}.} 
  \label{fig:PS}
  \hspace*{-0.02\textwidth}
\end{figure}

\section{$H(\Div)$-conforming HDG Discretization and Analysis}
\label{sec2:disc}
In the remainder of this paper we consider a special class of discretizations for linear elasticity: $H(\Div)$-conforming HDG discretizations where we also keep track of the volume-locking and gradient-robustness property of the method. In Subsections \ref{sec:hdivhdg:prelim} -- \ref{sec:hdivhdg:scheme} we introduce preliminaries, notation and the numerical method and analyse it with respect to quasi-optimal error estimates and volume-locking in Subsection \ref{sec:hdivhdg:errorest}. The prove of gradient-robustness is carried out in Subsection \ref{sec:hdivhdg:gradrobhdiv}. 
Numerical results support these theoretical findings in
Subsection \ref{sec:num1}.
In the subsequent section, Section \ref{sec:relaxedhdiv}, we consider a (more efficient) modified scheme which is volume-locking free, but is gradient-robust only after a simple modification.

\subsection{Preliminaries}\label{sec:hdivhdg:prelim}
Let $\Eh=\{F\}$ be the collection of facets (edges in 2D, faces in 3D) in $\Oh$.
We distinguish functions with support only on facets indicated by a subscript $F$ and those
with support also on the volume elements which is indicated by a subscript $T$. Compositions of both types
are used for the HDG discretization of the displacement and indicated by underlining, 
$\compU = (\bld u_T ,\bld u_F)$.
On each simplex $T$, 
we denote the tangential component of a vector $\bld v_T$ on a facet $F$
by $(\bld v_T)^t = \bld v_T-(\bld v_T\cdot \bld n)\bld n$, where $\bld n$ is the unit normal vector on $F$.
Furthermore, we denote the compound exact solution as $\compU:=(\bld u, \bld u^t)$, and introduce the composite space of sufficiently smooth functions
\begin{align}
\label{comp-space}
\compUspace: = [H^2_0(\Omega)]^d\times [H^1_0(\Eh)]^d. 
\end{align}

\subsection{Finite elements}\label{sec:hdivhdg:fem}
We consider an HDG method which approximates the displacement on the mesh 
$\Oh$ using an $H(\Div)$-conforming space and the {\it tangential } component of the displacement on 
the mesh skeleton $\Eh$ with a DG facet space given as follows:

\begin{subequations}
\label{space}
\begin{align}
\label{space-1}
\Sh : =&\; \{\bld v_T\in \prod_{T\in\Oh}[\pol^{k}(T)]^d: \;\;
\jmp{\bld v_T\cdot\bld n}_F = 0 \,\forall F\in\Eh\}\subset H_0(\Div,\Omega),\\
 M_{h} := &\;\{\bld v_F\in \prod_{F\in\Eh} M_k(F): \;\;
\bld v_F\cdot\bld n = 0 \,\forall F\in\Eh, \;\;
\bld v_F = 0 \,\;\;\forall F\subset \partial\Omega\},
\end{align}
where $\jmp{\cdot}$ is the usual jump operator, 
$\pol^k$ the space of polynomials up to degree $k$, and
\begin{align*}
 M_k(F):=\left\{
 \begin{tabular}{l l}
 $[\pol^{0}(F)]^3\oplus \bld x\times [\pol^{0}(F)]^3$ & if $k = 1$ and $d=3$,
  \vspace{.2cm}
\\
 $[\pol^{k-1}(F)]^d$ & else.
 \end{tabular}
 \right.
\end{align*}
Note that functions in $M_h$ are defined only on the mesh skeleton 
and have normal component {\it zero}.
%

%

To further simplify notation, we denote the composite space as
\[
 \compUhspace : = \Sh\times {M}_h.
\]
\end{subequations}

\subsection{The numerical scheme}\label{sec:hdivhdg:scheme}
We introduce the $L^2$ projection onto $M_k(F)$
$\Pi_M$:
 \[
  \Pi_M: [L^2(F)]^d\rightarrow M_k(F), 
  \quad \int_F (\Pi_M f) v \, \mathrm{ds} = \int_{F}f\,v\, \mathrm{ds} \quad \forall v\in M_k(F).
 \]
 Then, for all $\compU, \compV\in \compUhspace$, we introduce the bilinear and linear forms 
\begin{subequations}
\label{bilinearforms}
 \begin{align}
  \label{elas-1}
  a_{h}(\compU, \compV) :=&\; a_{h}^{\mu}(\compU, \compV) + a_{h}^{\lambda}(\compU, \compV)   \\
 \label{elas-2}
  a_{h}^{\mu}(\compU, \compV) :=&\; 
  \sum_{T\in\Oh}\int_T2\mu\,\gradss(\bld u_T):\gradss(\bld v_T)
  \dx\\
&\;  -\int_{\partial T}2\mu\,\gradss(\bld u_T)\bld n\cdot \jmp{\compV^t}\ds 
-\int_{\partial T}2\mu\,\gradss(\bld v_T)\bld n\cdot \jmp{\compU^t}\ds\nonumber\\
&\;  + \int_{\partial T}\mu\frac{\alpha}{h}\Pi_M\jmp{\compU^t}\cdot \Pi_M\jmp{\compV^t}\ds,
                                                                                    \nonumber\\
   a_h^{\lambda}(\compU, \compV) :=&\;
\sum_{T\in\Oh}\int_T  \lambda\,\Div(\bld u_T)\Div(\bld v_T)
                                     \dx,
 \label{elas-2b}   \\                       
  \label{elas-3}
f(\compV) := &\;   \sum_{T\in\Oh}\int_{T}\bld f\cdot\bld v_T\dx.
 \end{align}
 where 
 $\jmp{\compU^t}= (\bld u_T)^t-\bld u_F$ is the (tangential) jump between element interior and facet unknowns, and 
 $\alpha = \alpha_0 k^2$ with $\alpha_0$ a sufficiently large positive constant.
\end{subequations}
 
The numerical scheme then reads: Find $\compUh \in \compUhspace$ such that 
\begin{align}\tag{S1}
\label{scheme}
a_h(\compUh, \compVh) = f(\compVh), \quad \forall \compVh \in \compUhspace.
\end{align}

\subsection{Error estimates}\label{sec:hdivhdg:errorest}
We write
\[
 A\preceq B
\]
to indicate that there exists a constant $C$, independent of the mesh size $h$, the 
Lam\'e parameters $\mu$ and $\lambda$, 
and the numerical solution, such that 
$A\le CB.$

Denote the space of rigid motions
$$RM(T)=\{\bld a+B \,\bld x:\;\;\bld a\in \bR^d, B\in S_d\},$$
where $S_d$ is the space of anti-symmetric $d\times d$ matrices.
We observe that the tangential trace on a facet $F$ of any function in $RM(T)$ 
is a constant in 2D, and lies in the space $M_1(F)$ in 3D.
Hence,  there holds 
\begin{align}
\label{inclusion} 
\bld v^t|_F \in M_k(F), \quad\quad \forall \bld v\in RM(T).
\end{align}
The above property is the key to prove coercivity of the bilinear form \eqref{elas-1}.

We use the following projection $\Pi_{RM}$ from $[H^1(T)]^d$ onto $RM(T)$ \cite{Brenner03}:
\begin{alignat*}{2}
\int_T \Pi_{RM} \bld u \dx =&\;  \int_T \bld u \dx,\\
 \int_T \curls\, (\Pi_{RM} \bld u)\dx =&\;  
 \int_T \curls\,\bld u \dx,
\end{alignat*}
where $\curls\, \bld u$ is the anti-symmetric part of the gradient of $\bld u$. 
Following \cite{Brenner03} this projection operator has the approximation properties
\begin{subequations}
\label{proj-rm}
\begin{align}
\label{proj-rm-1}
 ||\grads(\bld u-\Pi_{RM}\bld u)||_{T} \preceq &\; 
  ||\gradss(\bld u)||_{T},\\
\label{proj-rm-2}
||\bld u-\Pi_{RM}\bld u||_{T}\preceq&\;
    h_T||\grads(\bld u-\Pi_{RM}\bld u)||_{T}.
\end{align}
\end{subequations}


Denoting the following  (semi)norms
\begin{subequations}
\label{norms}
\begin{equation*}
  \|\compV\|_{\mu,h} := \mu^\frac12 \|\compV\|_{1,h},\quad
  \|\compV\|_{\mu,*,h} := \mu^\frac12 \|\compV\|_{1,*,h},\quad
  \|\compV\|_{\mu,**,h} := \mu^\frac12 \|\compV\|_{1,**,h}, 
\end{equation*}
\begin{align}
 \label{norm-energy}
\|\compV\|_{1,h} := &\;
 \left(
\sum_{T\in\Oh} 2\|\gradss \bld v_T\|^2_T
+\frac{2}{h}\|\Pi_M\jmp{\compV^t}\|^2_{\partial T}
 \right)^{1/2},\\
 \label{norm-energy2}
\|\compV\|_{1,*,h} := &\;
\Big(
\|\compV\|_{1,h}^2+
\sum_{T\in\Oh}
2 h
\|\gradss(\bld v_T)\bld n\|^2_{\partial T}
 \Big)^{1/2},\\
 \label{norm-energy3}
 \|\compV\|_{1,**,h} := &\;
\Big(
\|\compV\|_{1,*,h}^2+\sum_{T\in\Oh}
\frac{2}{h}\|\jmp{\compV^t}\|^2_{\partial T}
 \Big)^{1/2}.
\end{align}
\end{subequations}
We also denote the $H^{s}$-norm  on $\Omega$ as $\|\cdot\|_{s}$, and when 
$s=0$, we simply denote $\|\cdot\|$ as the $L^2$-norm on $\Omega$.

To derive optimal $L^2$ error estimates, we shall assume the following full $H^2$-regularity
 \begin{align}
 \label{dual}
\mu \|\bld \phi\|_{2}
+
\lambda \|\Div\,\bld \phi\|_{1}
\preceq \|\bld \theta\|
 \end{align}
  for the dual problem with any source term $\bld \theta\in [L^2(\Omega)]^d$:
 \begin{subequations}
 \label{dual-eq}
 \begin{align}
- \Div\left(2\mu\gradss \bld \phi\right)  -  \grads\left(\lambda\,\Div\bld \phi\right)  =&\; \bld \theta
\quad \text{ in }  \Omega,\\
  \bld\phi =&\; \bld 0 \quad \text{ on }   \partial\Omega.
 \end{align}
 \end{subequations}
The estimate \eqref{dual} holds on convex polygons \cite{BrennerSung92}.

We have the following estimates.
\begin{theorem}
\label{thm:energy}
 Assume $k\ge 1$ and the regularity $\bld u\in [H^{k+1}(\Omega)]^d$.
 Let $\compUh\in \compUhspace$ be the numerical solution to the scheme \eqref{scheme}.
 Then, for sufficiently large stabilization parameter $\alpha_0$, 
 the following estimate holds
 \begin{subequations}
 \label{est}
 \begin{alignat}{2}
 \label{est-1}
\|\compU-\compUh\|_{\mu,h}
\preceq&\; \mu^{1/2} h^{k}
 \|\bld u\|_{k+1},\\
 \label{est-1b}
\|\Div(\bld u-\bld u_T)\|
\preceq &\;(\mu/\lambda)^{1/2} h^{k}
 \|\bld u\|_{k+1} + h^{k}\|\Div\,\bld u\|_{k}.
 \end{alignat}
Moreover, under the regularity assumption \eqref{dual}, the following estimate holds
 \begin{alignat}{2}
 \label{est-2}
 \|\bld u-\bld u_T\|
\preceq 
h^{k+1} \|\bld u\|_{k+1}.
 \end{alignat}
  \end{subequations}

\end{theorem}

\begin{remark}[Volume-locking-free estimates]
From the energy estimates \eqref{est-1}, we get that
\[
 \sum_{T\in\Oh} \|\gradss (\bld u-\bld u_{h,T})\|^2_T
+\frac{1}{h}\|\Pi_M\jmp{(\compU-\compUh)^t}\|^2_{\partial T}
 \preceq h^{2 k}
 \|\bld u\|_{k+1}^2,
\]
with the hidden constant independent of the Lam\'e constants $\lambda$ and $\mu$.
This observation also holds for the $L^2$-norm estimate \eqref{est-2}.
Hence, the estimates are free of volume-locking when $\lambda\rightarrow +\infty$.
\end{remark}

\begin{proof}
We proceed in the following five steps. 

\underline{Step 1 (Coercivity):} 
Observing the definition \eqref{elas-1} for the bilinear form $a_h^\mu(\cdot,\cdot)$, 
and applying the Cauchy-Schwarz inequality combined with trace-inverse inequalities, we obtain, cf. \cite[Lemma 2]{FL_JSC_2018}, for sufficiently large $\alpha$, 
\begin{equation} \label{eq:coercivity}
\|\compVh\|_{\mu,h}^2\preceq a_h^{\mu}(\compVh,\compVh) \quad \forall \compVh\in \compUhspace.
\end{equation}

\underline{Step 2 (Norm equivalence):} 
By property \eqref{inclusion}, we have 
$\Pi_M(\Pi_{RM}\bld v_T)^t = (\Pi_{RM}\bld v_T)^t$.
Hence, for any interior facet $F\in \Eh\backslash\partial\Omega$ and 
any function $\compV\in \compUspace+\compUhspace$, we have
\begin{align*}
 \|\jmp{\compV^t}\|_{F}\le&\;
  \|\Pi_M\jmp{\compV^t}\|_{F}+
   \|\bld v_T^t-\Pi_M\bld{v}_T^t\|_{F}\\
   \le &\;
   \|\Pi_M\jmp{\compV^t}\|_{F}+
   \|(\bld v_T-\Pi_{RM}\bld v_T)^t-\Pi_M(\bld{v}_T-\Pi_{RM}\bld v_T)^t\|_{F}\\
   \preceq
   &\;
   \|\Pi_M\jmp{\compV^t}\|_{F}+
   \|\bld v_T-\Pi_{RM}\bld v_T\|_{F}.
\end{align*}
Using the trace theorem and approximation properties \eqref{proj-rm} of the projector $\Pi_{RM}$, we get 
\begin{align*}
    \|\bld v_T-\Pi_{RM}\bld v_T\|_{F}^2
    \preceq&\;
    \sum_{T\in \mathcal{T}(F)}(h|\bld v_T-\Pi_{RM}\bld v_T|^2_{1,T}+h^{-1}\|\bld v_T-\Pi_{RM}\bld v_T\|^2_T)\\    
    \preceq&\;
   h\, \|\gradss \bld v_T\|^2_{\mathcal{T}(F)},
\end{align*}
where $\mathcal{T}(F)$ is the set of the two simplexes meeting $F$.
Hence, 
\begin{align}
\label{jump-est}
 \|\jmp{\compV^t}\|_{F}\le&\;
   \|\Pi_M\jmp{\compV^t}\|_{F} +   h^{1/2}\, \|\gradss \bld v_T\|_{\mathcal{T}(F)}
   \quad\quad \forall \compV\in \compUspace+\compUhspace.
\end{align}
Recally the norms defined in \eqref{norms},
this directly implies
\begin{subequations}
 \label{norm-eq}
\begin{align}
 \label{norm-eq1}
 \|\compV\|_{\mu,**,h}\preceq
  \|\compV\|_{\mu,*,h}\quad\quad\forall \compV\in \compUspace+\compUhspace.
\end{align}
On the other hand, by trace and inverse inequalities, we have, cf. \cite[Lemma 1]{FL_JSC_2018},
\begin{align}
 \label{norm-eq2}
 \|\compVh\|_{\mu,*,h}\preceq
  \|\compVh\|_{\mu,h}\quad\quad\forall \compVh\in \compUhspace.
\end{align}
\end{subequations}

\underline{Step 3 (Boundedness):} 
Applying the Cauchy-Schwarz inequality on the bilinear form $a_h(\cdot,\cdot)$, we obtain
using the estimate \eqref{jump-est}
\begin{align}
\label{bdd-est}
a_h^\mu(\compV,\compW)
\le&\; 
\|\compV\|_{\mu,**,h}
\|\compW\|_{\mu,**,h}
\preceq\;\|\compV\|_{\mu,*,h}
\|\compW\|_{\mu,*,h}\quad\quad
 \forall \compV,\compW\in \compUspace+\compUhspace.
\end{align}


\underline{Step 4 (Galerkin orthogonality, BDM interpolation): }
Galerkin orthogonality yields 
$a_h(\compU, \compVh) = f(\compVh)$ for all $\compVh\in \compUhspace$. 
Hence, 
$a_h(\compU-\compUh, \compVh) = f(\compVh)$.
We estimate the error by first applying a triangle inequality to split
$$
\|\compU-\compUh\|_{\mu,h} \leq \|\compVh-\compU\|_{\mu,h} + \|\compUh-\compU\|_{\mu,h},
$$
where we choose $\compVh = (\Pi_V \bld u, \Pi_M \bld u)$ where $\Pi_V$ is the classical BDM interpolator, \cite[Proposition 2.3.2]{BoffiBrezziFortin13}. 
We note that the interpolation operator $\Pi_V$ has, as a consequence of its commuting diagram property, that 
$$
 \int_{\Omega} \Div (\Pi_V \bld u - \bld u) q_h \dx =  \int_{\Omega} (\Pi_Q \Div \bld u - \Div \bld u) q_h \dx = 0 \quad \forall~ q_h \in Q_h,
 $$
where $\Pi_Q$ is the $L^2$ projection into $Q_h = \prod_{T\in\Oh}\pol^{k-1}(T) = \Div \Sh$.
Hence,
\begin{align*}
&  \|\compUh-\compVh\|_{\mu,h}^2
  +\lambda \|\Div(\bld u_T-\bld v_T)\|^2\\
  & \preceq a_h^{\mu}(\compUh-\compVh, \compUh - \compVh)
  +\lambda \|\Div(\bld u_T-\bld v_T)\|^2\\
& = a_h(\compUh-\compVh, \compUh - \compVh) 
 = a_h(\compU-\compVh, \compUh - \compVh) \\
  &= a_h^{\mu}(\compU-\compVh,\; \compUh - \compVh)
    + \underbrace{a_h^{\lambda}(\compU-\compVh,\; \compUh - \compVh)}_{=0}
 \\[-1ex]
&\preceq  \|\compU-\compVh\|_{\mu,\ast,h} \|\compUh-\compVh\|_{\mu,\ast,h} 
\preceq  \|\compU-\compVh\|_{\mu,\ast,h} \|\compUh-\compVh\|_{\mu,h}.
\end{align*}
This implies
\begin{align}
\label{iii}
\|\compU-\compUh\|_{\mu,h}
  +\lambda^{1/2}\|\Div(\bld u_T-\bld v_T)\|
  \preceq  \|\compU-\compVh\|_{\mu,\ast,h}
  \preceq  \mu^{1/2} h^k\|\bld u\|_{k+1},  
\end{align}
where the last estimate follows from usual Bramble-Hilbert-type arguments, cf. \cite[Proposition 2.3.8]{Lehrenfeld:10} 
for a proof in an almost identical setting. 
The estimate \eqref{est-1} follows directly from \eqref{iii}, and
the estimate \eqref{est-1b} follows from \eqref{iii} and the triangle inequality:
\begin{align*}
 \|\Div(\bld u-\bld u_T)\|\le &\;
  \|\Div(\bld u_T-\bld v_T)\| +
 \underbrace{\|\Div(\bld u-\bld v_T)\|}_{=\|(I-\Pi_Q)\Div\,\bld u\|}
    \\
    \preceq 
    &\; 
    (\mu/\lambda)^{1/2} h^k\|\bld u\|_{k+1}
    +h^k \|\Div\,\bld u\|_k.  
\end{align*}

\underline{Step 5 (Duality):} 
Let $\bld \phi$ be the solution to the dual problem \eqref{dual-eq} with 
$\bld \theta = \bld u- \bld u_T$ and $\compS = (\bld  \phi, \bld \phi^t) \in \compUspace$.
By symmetry of the bilinear form $a_h(\cdot,\cdot)$ and consistency of the numerical scheme \eqref{scheme}, we have with $\compSh = (\Pi_V \bld  \phi, \Pi_M \bld \phi) \in \compUhspace$
 \begin{align*}
   \|\bld u- \bld u_T\|_\Omega^2 &=
a_h(\compS, \compU-\compUh) = 
a_h(\compS-\compSh, \compU-\compUh)
\\
    & =
a_h^{\mu}(\compS - \compSh, \compU-\compUh) + 
a_h^{\lambda}(\compS - \compSh, \compU-\compUh) \\
    & =
a_h^{\mu}(\compS - \compSh, \compU-\compUh) + 
\lambda \sum_{T\in\Oh} \int_T \underbrace{\Div(\bld \phi- \Pi_V \bld  \phi) \Div(\bld u - \Pi_V\bld u)}_{
=(I-\Pi_Q)\Div\bld \phi\, (I-\Pi_Q)\Div\bld u} \dx \\
& \preceq 
\|\compS-\compSh\|_{\mu,*,h}
\|\compUh-\compU\|_{\mu,*,h}+ \lambda \|(I-\Pi_Q)\Div\bld \phi\|\cdot\|(I-\Pi_Q)\Div\bld u\|\\
&\preceq
          \mu h^{k+1} \|\bld\phi\|_2\|\bld u\|_{k+1}+
          \lambda h^{k+1} \|\Div\bld\phi\|_1\|\Div\bld u\|_{k}
          \\
& \preceq 
          h^{k+1} \|\bld u- \bld u_T\|_\Omega \|\bld u\|_{k+1},
\end{align*}
In the last step we invoked the regularity assumption \eqref{dual}. 
This completes the proof of \eqref{est-2}.
\end{proof}

\subsection{Gradient-robustness} \label{sec:hdivhdg:gradrobhdiv}
In this subsection we want to show that the $H(\Div)$-conforming HDG method in \eqref{scheme} is \emph{gradient-robust}. In this section a splitting into a discretely divergence-free subspace and an orthogonal complement is crucial. To proceed, it seems more natural to work with a DG-equivalent reformulation of the HDG scheme \eqref{scheme} by eliminating the facet unknowns (for analysis purposes only).  In Remark \ref{rem:splittinghdg} below we explain how this translate to the HDG setting.

We introduce the lifting $\mathcal{L}_h: \Sh + [H^2_0(\Omega)]^d \to M_h$ where $\mathcal{L}_h(\bld w_T)$ is the unique function in $M_h$ such that
\[
a_h((\bld w_T, \mathcal{L}_h(\bld w_T), (0, \bld v_F))=0,\quad \forall \bld v_F\in M_h.
\]
For the case of a uniform mesh size $h$, an explicit formula can easily derived yielding
\[
\mathcal{L}_h(\bld w_T)= \{\!\!\{ \Pi_M \bld w_T \}\!\!\}_* - \frac{h}{2 \alpha} \jmp{\gradss \bld w_T \cdot \bld n }_*,
\]
  where $\{\!\!\{ \cdot \}\!\!\}_*$ and $\jmp{\cdot}_*$ are the usual DG average and jump operators.
Then the solution $\underline{\bld u}_h=(\bld u_T, \bld u_F)\in \compUhspace$ to the scheme \eqref{scheme}
satisfies $\bld u_F = \mathcal{L}_h(\bld u_T)$, with $\bld u_T\in \Sh$ being the 
unique function  such that
\begin{equation}\tag{S1-DG}\label{eq:S1-DG}
  \hat a_h(\bld u_T,\bld v_T) = \hat{f}(\bld v_T) \quad \forall \bld v_T\in \Sh,
\end{equation}
where
$\hat a_h(\cdot,\cdot)$ and $\hat f$ are defined on $\Sh$ as follows:
\[
\hat a_h(\bld v_T,\bld w_T) :=  a_h\left((\bld v_T, \mathcal{L}_h(\bld v_T)),(\bld w_T, 0)\right), \quad \hat{f}(\bld w_T) := f((\bld w_T, 0)), ~\bld v_T, \bld w_T\in \Sh.
\]
%
Analogously (with slight abuse of notation) we define a norm on $\Sh$ with
$$
\Vert \bld u_T \Vert_{1,h} := 
\Vert ( \bld u_T, \mathcal{L}_h(\bld u_T) ) \Vert_{1,h}.
$$
Introducing the spaces
\begin{subequations}
  \label{eq:def:decomp}
  \begin{align}
 \Sh^0 := \{ \bld v_T \in \Sh : \Div \bld v_T = 0, \quad \forall T \in \Oh \},
\end{align}
and 
\begin{align}
 \Sh^\perp := \{ \bld v_T \in \Sh : \hat{a}_h(\bld v_T, \bld w_T) = 0, \quad \forall \bld w_T \in \Sh \}.
\end{align}
\end{subequations}
We then split the solution $\bld u_T \in \Sh$ to the scheme \eqref{eq:S1-DG} as $\bld u_T = \bld u_T^0 +\bld u_T^\perp $
where $\bld u_T^0,\bld u_T^\perp \in \Sh$ are the unique solutions to the following equations:
    \begin{subequations} \label{eq:decomp}
  \begin{align}
    \hat a_{h}(\bld u_T^0, \bld v_T^0) & = \hat f(\bld v_T^0) \quad \forall ~ \bld v_T^0 \in \Sh^0, \label{eq:auh0vh0}\\
    \hat a_{h}(\bld u_T^\perp, \bld v_T^\perp)
  & = \hat f(\bld v_T^\perp) \quad \forall ~ \bld v_T^\perp\in \Sh^\perp.
        \label{eq:auhpvhp}
  \end{align}
\end{subequations}
We are now ready to state the following gradient-robustness property of the schemes \eqref{eq:S1-DG} and\eqref{scheme} analogously to the continuous case in Theorem \ref{lem:gradient:field:robustness}.
\begin{theorem}[Gradient-robustness of \eqref{eq:S1-DG}]\label{thm:gradientrobustnessh}
  The scheme \eqref{eq:S1-DG} (and hence scheme \eqref{scheme}) is \emph{gradient-robust}, i.e. for $\bld f=\nabla \phi$, $\phi \in H^1(\Omega)$, the solution $\bld u_T
  = 
  \bld u_T^0
  +
  \bld u_T^\perp
 \in \Sh$ satisfies
  $$
\bld u_T^0= 0 ,\quad 
  \bld u_T^\perp
  = \mathcal{O}(\lambda^{-1}).$$
In particular, for $\lambda \to \infty$ one gets
  $\bld u_T \to \bld 0$.
\end{theorem}
To prove Theorem \ref{thm:gradientrobustnessh}, we shall first recall the following 
inf-sup stability result.
\begin{lemma}[inf-sup stability] \label{lem:infsupstokes}
  \begin{subequations}
    The following properties hold:\\ 
    There holds the discrete LBB condition:
    \begin{equation}\label{eq:lbb}
      \sup_{\bld u_T  \in \Sh} (\Div \bld u_T, q_h) \geq \beta \Vert q_h \Vert_{L^2(\Omega)} \Vert \bld u_T \Vert_{1,h} \quad \text{for all } q_h \in Q_h.
    \end{equation}
    for $\beta$ independent of $\mu,~h,~k$.
    Moreover, for all $q_h \in Q_h$ there exists a unique $\bld u_T^\perp \in \Sh^\perp$%
, s.t.  
    \begin{equation} \label{eq:lbb2}
      \Div(\bld u_T^\perp) = q_h \quad \text{ and } \quad \Vert \bld u_T^\perp  \Vert_{1,h} \leq \beta^{-1} \Vert q_h \Vert_{L^2(\Omega)}. 
    \end{equation}
    \end{subequations}
\end{lemma}
\begin{proof}
  For \eqref{eq:lbb} we refer to \cite{lederer2017polynomial} where \eqref{eq:lbb2} is a direct consequence of \eqref{eq:lbb} as it implies the existence of an isomorphism between $\Sh^\perp$ and $Q_h$ related to $(\Div(\cdot),\cdot)$,  cf. e.g. \cite[Lemma 3.58]{john2016finite}. \end{proof}
We now prove Theorem \ref{thm:gradientrobustnessh}.
\begin{proof}[Proof of Theorem \ref{thm:gradientrobustnessh}]
  With $\hat f(\cdot) = (\nabla \phi, \cdot)_{\Omega}$
  there holds after partial integration
\begin{equation}\label{eq:fvh0}
\hat f(\bld v_T^0)= - \sum_{T \in \mathcal{T}_h} (\phi, \Div \bld v^0_T)_T + \sum_{F \in \mathcal{F}_h} ( \phi, \jmp{\bld v^0_T \cdot \bld n}_F ) = 0 \quad \forall ~ \bld v_T^0 \in \Sh^0.
\end{equation}
  From the decomposition in \eqref{eq:decomp}  we hence have $\bld u_T^0=0$.
  Taking $\bld v_T^\perp := 
  \bld u_T^\perp$  in \eqref{eq:auhpvhp}
we get
    \begin{equation*}
   \mu \Vert\bld u_T^\perp \Vert_{1,h}^2
   +\lambda  \Vert\Div({{\bld u}}_T^\perp) \Vert^2
 \preceq 
\hat a_h(\bld u_T^\perp ,\bld u_T^\perp)
= \hat f (\bld u_T^\perp) \preceq \Vert \phi \Vert_{H^{1}(\Omega)} \Vert \bld u_T^\perp \Vert_{1,h}.
    \end{equation*}
Since Lemma \ref{lem:infsupstokes} implies that 
    \[
 \Vert \bld u_T^\perp \Vert_{1,h} \leq \beta^{-1} \Vert \Div(\bld u_T^\perp) \Vert,
    \]
    we finally obtain
    \begin{equation*}
      \Vert  \bld u_T^\perp \Vert_{1,h} \preceq \frac{1}{\mu + \lambda} \Vert \phi \Vert_{1} \quad \stackrel{\lambda \to \infty}{\longrightarrow} 0.
    \end{equation*}
  \end{proof}
  \begin{remark}\label{rem:splittinghdg}
    The splitting into a divergence-free subspace and its $a_h$-orthogonal complement can also be done for $\compUhspace$. Let us relate the splitting of $\Sh$ to a corresponding splitting of $\compUhspace$.
    First, there holds $\compUhspace^0 = \Sh^0 \times M_h$ and
    $\compUhspace^\perp = \{ (\bld v_T, \bld v_F) \in \compUhspace \mid \bld v_T \in \Sh^\perp, \bld v_F = \mathcal{L}_h(\bld v_T)\}$. Second, the solution $\compUh$ of \eqref{scheme} then has the splitting $\compUh = \compUh^0 + \compUh^\perp$ with 
    $\compUh^0 = (\bld u_T^0,\mathcal{L}_h(\bld u_T^0)) \in \compUhspace^0$ 
    and $\compUh^\perp = (\bld u_T^\perp,\mathcal{L}_h(\bld u_T^\perp)) \in \compUhspace^\perp$ and for $\bld f = \nabla \phi,~\phi \in H^1(\Omega)$ there holds $\compUh^0 = 0$ and $\compUh^\perp = \mathcal{O}(\lambda^{-1})$.
  \end{remark}

  \subsection{Numerical results}
\label{sec:num1}
The numerical results for the 
two examples in Section \ref{sec:mot} 
for the scheme \eqref{scheme} are given in Figure \ref{fig:H1ab} and are consistent with 
the results in Theorem \ref{thm:energy} and Theorem
\ref{thm:gradientrobustnessh}.

\begin{figure}
  \begin{center}
    \includegraphics[height=0.255\textwidth]{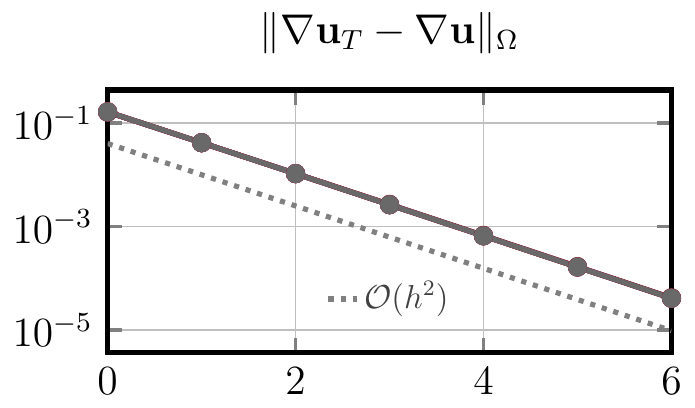}
    \hspace*{-0.02\textwidth}
\includegraphics[height=0.255\textwidth]{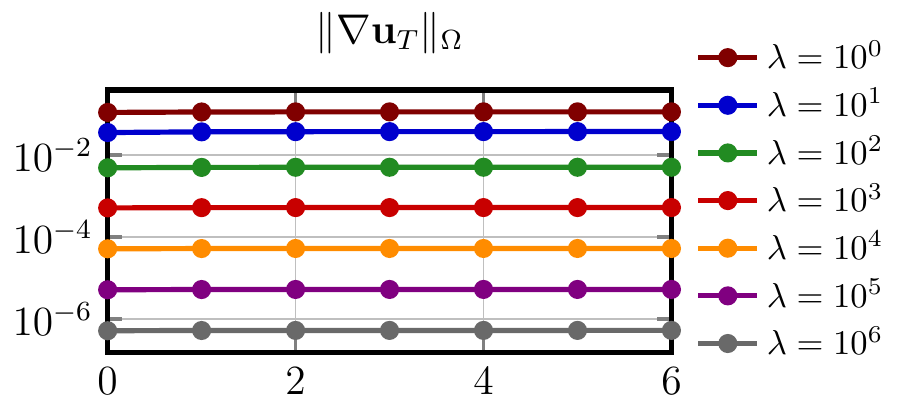}
    \hspace*{-0.06\textwidth}
  \end{center}    
  \caption{Discretizaton error for Example \ref{ex:1} (left) and norm of discrete error for Example \ref{ex:2} (right) for the method \eqref{scheme}, $k=2$, on a barycentric-refined mesh under mesh refinement ($x$-axis: refinement level $L$) and different values of $\lambda$ for Example \ref{ex:2}.} 
  \label{fig:H1ab}
  \hspace*{-0.02\textwidth}
\end{figure}


\section{Relaxed $H(\Div)-$conforming HDG discretization} \label{sec:relaxedhdiv}
The results in Theorem \ref{thm:energy} provide optimal error estimates for the method \eqref{scheme}.
However, for the approximation of the displacement with a polynomial degree $k$ it requires unknowns of degree $k$
for the normal component of the displacement on every facet of the mesh. 
In view of the superconvergence property of other HDG
methods \cite{QiuShenShi16,CockburnFu17c}, where
only unknowns of polynomial degree $k-1$ on the facets are required to obtain an accurate polynomial
approximation of order $k$ (possibly after a local post-processing) this is sub-optimal. 
Here we follow \cite{LLS_ARXIV_2017} to slightly relax the $H(\Div)$-conformity so that only unknowns of polynomial degree 
$k-1$ are involved for normal-continuity. This allows for optimality of the method also in the sense of superconvergent HDG methods.
The resulting method is still volume-locking-free.
We assume the polynomial degree $k\ge 2$ in the following discussion.

\subsection{The relaxed $H(\Div)$-conforming HDG scheme}
We introduce the modified vector space 
\begin{align}
 \label{m-space-1}
 \Sh^- : =&\; \{\bld v_T\in \prod_{T\in\Oh}[\pol^{k}(T)]^d: \;\;
\Pi_F^{k-1}\jmp{\bld v_T\cdot\bld n}_F = 0,\; \,\forall F\in\Eh\},
\end{align}
where $\Pi_F^{k-1}: L^2(F)\rightarrow P^{k-1}(F)$ is the $L^2(F)$-projection:
\begin{align}
 \label{proj-f}
 \int_F(\Pi_F^{k-1}w)v\,\mathrm{ds} = \int_Fw\,v\,\mathrm{ds},\quad\forall v\in P^{k-1}(F).
\end{align}
Details of the construction of the finite element space $\Sh^-$ can be found in \cite[Section 3]{LLS_ARXIV_2017}.
Functions in $\Sh^-$ are only ``almost normal-continuous'', but can be normal-discontinuous in the highest orders. 

Denoting the compound finite element space 
\begin{align*}
 \compUhspace^{\,-}:=  \Sh^-\times M_h,
\end{align*}
then the relaxed $H(\Div)$-conforming HDG scheme reads: Find $\compUh\in \compUhspace^{\,-}$ such that 
\begin{align}\tag{S2}
 \label{scheme-2}
 a_h(\compUh,\compVh) = f(\compVh), \quad \forall \compVh\in \compUhspace^{\,-}.
\end{align}
\begin{remark}
Notice that the globally coupled degrees of freedom for the above 
relaxed $H(\Div)$-conforming scheme are polynomials of degree $k-1$ per facet for both 
tangential and normal component of the displacement, while that for the original $H(\Div)$-conforming scheme \eqref{scheme}
are polynomials of degree $k-1$ per facet for the tangential component of the displacement, and 
polynomials of degree $k$ per facet for the normal component. 
This relaxation reduces the globally coupled degrees of freedom  which improves the sparsity pattern of the
linear systems.
\end{remark}

\subsection{Error estimates}
The error analysis of the relaxed scheme \eqref{scheme-2} follows closely from that for the original scheme \eqref{scheme} in
Theorem \ref{thm:energy}.

Due to the violation of $H(\Div)$-conformity of $\Sh^-$, we have a consistency term to take care of.

\begin{lemma}
 \label{lemma:cs}
 Let $\bld u\in [H_0^2(\Omega)]^d$ be the solution to the equations \eqref{eqns} and define the splitting $\bld f = \bld f^\mu + \bld f^\lambda$ with $\bld f^\mu = - \Div\left(2\mu\gradss \bld u\right)$ and $\bld f^\lambda = -  \grads\left(\lambda\,\Div \bld u\right)$ and $f(\cdot) = f^{\mu}(\cdot) + f^{\lambda}(\cdot)$ correspondingly.
 Denote $\compU:=(\bld u, \bld u^t)\in \compUspace$.
 There holds for all $\compV=(\bld v_T,\bld v_F)\in \compUhspace^{\,-} + \compUspace$
 \begin{subequations}
 \begin{align}
  \label{csmu}
  a_h^{\mu}(\compU, \compV) &= f^\mu(\compV) + \mathcal{E}_c^\mu(\bld u, \compV),\\
  \label{cslambda}
  a_h^\lambda(\compU, \compV) &= f^\lambda(\compV) + \mathcal{E}_c^\lambda(\bld u, \compV), \\
  \label{cs}
  a_h(\compU, \compV) &= f(\compV) + \mathcal{E}_c(\bld u, \compV), \\
                                                   \intertext{with}
\mathcal{E}_c^{\mu}(\bld u, \compV)
& =
\sum_{T\in\Oh}\int_{\partial T}\left(2\mu(\gradss(\bld u)\bld n)\cdot\bld n\right)
                                (id-\Pi_F^{k-1})(\bld v_T\cdot\bld n). \\
\mathcal{E}_c^{\lambda}(\bld u, \compV)
& =
\sum_{T\in\Oh}\int_{\partial T}\left(\lambda\Div\bld u\right)
                                (id-\Pi_F^{k-1})(\bld v_T\cdot\bld n), \\
\mathcal{E}_c(\bld u, \compV) &= \mathcal{E}_c^{\mu}(\bld u, \compV) + \mathcal{E}_c^{\lambda}(\bld u, \compV).
  \label{cs-term}
 \end{align}
\end{subequations}
Moreover, for $\bld u\in [H_0^\ell(\Omega)]^d$, $\ell\ge 2$ and $1\le m\le\min(k,\ell-1)$ we have
\begin{subequations}
\begin{align}
\mathcal{E}_c^\mu(\bld u, \compV)
& \preceq 
h^m \mu^{1/2}\|\bld u\|_{m+1} \|\compV\|_{\mu,h}, \quad
\mathcal{E}_c^\lambda(\bld u, \compV)
\preceq 
h^m \frac{\lambda}{\mu^{1/2}}\|\Div\bld u\|_m \|\compV\|_{\mu,h}.
\end{align}
\begin{align}
\mathcal{E}_c(\bld u, \compV)
& \preceq 
h^m\left(\mu^{1/2}\|\bld u\|_{m+1}+\frac{\lambda}{\mu^{1/2}}\|\Div\bld u\|_m\right)\|\compV\|_{\mu,h}.
  \label{cs-term-est}
\end{align}
\end{subequations}
\end{lemma}
\begin{proof}
By continuity of $\bld u$ and integration by parts, we get
 \begin{align*}
    a_h^{\mu}(\compU, \compV)-f^{\mu}(\compV) 
    = &\; \sum_{T\in \Oh}
 \int_{\partial T}2\mu\,\gradss(\bld u)\bld n\cdot (\bld v_T-\bld v_T^t)\ds \\
= &\; \sum_{T\in \Oh}
 \int_{\partial T} 2\mu(\gradss(\bld u)\bld n\cdot\bld n 
(\bld v_T\cdot\bld n)\,\ds\\ 
= &\; \sum_{T\in \Oh}
 \int_{\partial T}(2\mu(\gradss(\bld u)\bld n)\cdot\bld n) 
(id-\Pi_F^{k-1})(\bld v_T\cdot\bld n)\,\ds\\ 
   = &\; \mathcal{E}_c^{\mu}(\bld u,\compV),
 \end{align*}
where the third equality follows from the fact that $\Pi_{F}^{k-1}\jmp{\bld v\cdot\bld n}_F = 0$ for all
$\bld v\in \Sh^-$. Analogously we obtain $a_h^{\lambda}(\compU, \compV)-f^{\lambda}(\compV) =  \mathcal{E}_c^{\lambda}(\bld u,\compV)$.

Applying the Cauchy-Schwarz inequality and properties of the $L^2$-projection, we have
\begin{align*}
\mathcal{E}_c^{\mu}(\bld u,\compV) = &\int_{\partial T}(id-\Pi_F^{k-1})\left(2\mu(\gradss(\bld u)\bld n)\cdot\bld n\right)
(id-\Pi_F^{k-1})(\bld v_T\cdot\bld n)\\
 &\;\le 
 \left(2\mu\|(id-\Pi_F^{k-1})\gradss(\bld u)\|_{\partial T}\right)
 \|(id-\Pi_F^{k-1})(\bld v_T\cdot\bld n)\|_{\partial T}\\
  &\;\preceq
h^{m-1/2} \mu\|\gradss(\bld u)\|_{H^m(T)} 
 \|(id-\Pi_F^{k-1})(\bld v_T\cdot\bld n)\|_{\partial T}\\
  &\;\preceq
h^m \mu\|\bld u\|_{H^{m+1}(T)}
 \|(id-\Pi_{RM})\bld v_T\|_{\partial T} \; \preceq
h^m\mu\|\bld u\|_{H^{m+1}(T)}
\|\gradss\bld v_{T} \|_{T},
\end{align*}
where the last inequality follows from the trace theorem and the approximation properties \eqref{proj-rm}.
Similarly,
\begin{align*}
\mathcal{E}_c^{\lambda}(\bld u,\compV) = &\int_{\partial T}(id-\Pi_F^{k-1}) \lambda\Div\bld u
(id-\Pi_F^{k-1})(\bld v_T\cdot\bld n)\\
 &\;\le 
 \lambda\|(id-\Pi_F^{k-1})\Div \bld u\|_{\partial T}
 \|(id-\Pi_F^{k-1})(\bld v_T\cdot\bld n)\|_{\partial T}\\
&\; \preceq
h^m
{ \lambda}\|\Div \bld u\|_{H^m(T)}
\|\gradss\bld v_{T} \|_{T}.
 \end{align*}
Summing over all elements concludes the proof.
\end{proof}

We have the following error estimates, whose proof follows closed from that for Theorem \ref{thm:energy}. 
We only sketch the proof with a focus on the modification needed from the proof for Theorem \ref{thm:energy}.
\begin{theorem}
\label{thm:2}
 Assume $k\ge 2$ and the regularity $\bld u\in [H^{k+1}(\Omega)]^d$.
 Let $\compUh\in \compUhspace^{\,-}$ be the numerical solution to the scheme \eqref{scheme-2}.
 Then, for sufficiently large stabilization parameter $\alpha_0$, 
 the following estimate holds
 \begin{subequations}
 \label{est2}
 \begin{alignat}{2}
 \label{est2-1}
\|\compU-\compUh\|_{\mu,h}
\preceq&\; h^{k}
 (\mu^{1/2} \|\bld u\|_{k+1}+\frac{\lambda}{\mu^{1/2}}\|\Div\bld u\|_k),\\
 \label{est2-1b}
\|\Div(\bld u-\bld u_T)\|
\preceq &\;(\mu/\lambda)^{1/2} h^{k}
 \|\bld u\|_{k+1} + \left(\frac{\lambda^{1/2}}{\mu^{1/2}}+1\right)h^{k}\|\Div\,\bld u\|_{k}.
 \end{alignat}
Moreover, under the regularity assumption \eqref{dual}, the following estimate holds
 \begin{alignat}{2}
 \label{est2-2}
 \|\bld u-\bld u_T\|
\preceq 
h^{k+1} \left(\|\bld u\|_{k+1}+(\frac{\lambda}{\mu}+1)\|\Div \bld u\|_{k}\right).
 \end{alignat}
  \end{subequations}
\end{theorem}

\begin{remark}[Volume-locking-free estimates]
For convex polygonal domain $\Omega$, it is proven \cite{BrennerSung92} that 
\[
 \mu\|\bld u\|_2+\lambda \|\Div \bld u\|_1\preceq \|\bld f\|.
\]
If we have the regularity shift, for $k\ge 2$,
\[
 \mu\|\bld u\|_{k+1}+\lambda \|\Div \bld u\|_{k}\preceq \|\bld f\|_k,
\]
the above estimates are free of volume-locking when $\lambda\rightarrow +\infty$.
\end{remark}

\begin{proof}
To prove the energy estimates \eqref{est2-1} and \eqref{est2-1b}, 
we still take $\compVh = (\Pi_V \bld u, \Pi_M \bld u)\in \compUhspace\subset \compUhspace^{\,-}$ 
as in the proof of Theorem \ref{thm:energy}.
By coercivity, 
\begin{align*}
&  \|\compUh-\compVh\|_{\mu,h}^2
  +\lambda \|\Div(\bld u_T-\bld v_T)\|^2\\
  & \preceq a_h^{\mu}(\compUh-\compVh, \compUh - \compVh)
  +\lambda \|\Div(\bld u_T-\bld v_T)\|^2\\
& = a_h(\compUh-\compVh, \compUh - \compVh) 
 = a_h(\compU-\compVh, \compUh - \compVh) -\mathcal{E}_c(\bld u,\compUh-\compVh)\\
 &= a_h^{\mu}(\compU-\compVh,\; \compUh - \compVh) 
  -\mathcal{E}_c(\bld u,\compUh-\compVh)
 \\
&\preceq \left( \|\compU-\compVh\|_{\mu,\ast,h} 
+\mu^{1/2}h^k\|\bld u\|_{k+1}+\frac{\lambda}{\mu^{1/2}}h^k\|\Div\bld u\|_k\right) \|\compUh-\compVh\|_{\mu,h}
\end{align*}
This implies
\begin{align*}
\|\compU-\compUh\|_{\mu,h}
  +\lambda^{1/2}\|\Div(\bld u_T-\bld v_T)\|
  \preceq  h^k\left(\mu^{1/2}\|\bld u\|_{k+1}+\frac{\lambda}{\mu^{1/2}}\|\Div\bld u\|_k\right).
\end{align*}
Then, the estimates \eqref{est2-1} and  \eqref{est2-1b} follows from \eqref{iii} and the triangle inequality.

To prove the $L^2$-estimate, let $\bld \phi$ be the solution to the dual problem \eqref{dual-eq} with 
$\bld \theta = \bld u- \bld u_T$ and $\compS = (\bld  \phi, \bld \phi^t) \in \compUspace$.
By symmetry of the bilinear form $a_h(\cdot,\cdot)$ and Lemma \ref{lemma:cs}, we have, with $\compSh = (\Pi_V \bld  \phi, \Pi_M \bld \phi) \in \compUhspace$
\begin{align*}
  \|\bld u- \bld u_T\|_\Omega^2 &=
a_h(\compS, \compU-\compUh) - \mathcal{E}_c(\bld \phi, \compU-\compUh)\\
&= 
a_h(\compS-\compSh, \compU-\compUh)
- \mathcal{E}_c(\bld \phi, \compU-\compUh)
+
\underbrace{\mathcal{E}_c(\bld u, \compSh)}_{=0}
\\
& \preceq 
h (\mu\|\bld \phi\|_2+\lambda\|\Div\bld \phi\|_1)
(\mu^{-1/2}\|\compU-\compUh\|_{\mu,*,h}+ \|(I-\Pi_Q)\Div\bld u\|)\\
& \preceq 
          h^{k+1} \|\bld u- \bld u_T\|_\Omega 
          \left(\|\bld u\|_{k+1}+(\frac{\lambda}{\mu}+1)\|\Div \bld u\|_{k}\right).
\end{align*}
In the last step we invoked the regularity assumption \eqref{dual}. 
This completes the proof of \eqref{est2-2}.
\end{proof}

\begin{remark}[Lack of gradient-robustness as a locking phenomenon]
Although, the scheme \eqref{scheme-2} is free of
volume-locking, it is not
free of another locking phenomenon, though.
Indeed, the explicit dependence of the right
side of the error estimate \eqref{est2-2}
on $\lambda$ indicates a classical locking
phenomenon in the sense of Babu\v{s}ka and Suri
\cite{BabuskaSuri1992}, where they write in the abstract:
``A numerical scheme for the approximation of a parameter-dependent problem is said to exhibit locking if the accuracy of the approximations deteriorates as the parameter tends to a limiting value.''
Comparing with the error
estimate \eqref{est-2} for the gradient-robust scheme
\eqref{scheme}, we recognize that
schemes for nearly-incompressible
linear elasticity are only
locking-free in the sense of \cite{BabuskaSuri1992},
if they are gradient-robust and free of volume-locking, simultaneously.
The situation is very similar to the incompressible Stokes problem.
Only schemes, which are pressure-robust and
discretely inf-sup stable simultaneously \cite{MR3743746}, are really locking-free
in the sense of Babu\v{s}ka and Suri \cite{BabuskaSuri1992}.
\end{remark}

\subsection{Numerical results for the scheme \eqref{scheme-2}}
The numerical results for the 
two examples in Section \ref{sec:mot} 
for the scheme \eqref{scheme-2} are given in Figure
\ref{fig:H2ab}. 
We observe from Figure \ref{fig:H2ab} (left) that 
the errors for the scheme \eqref{scheme-2} are
independent of $\lambda$ for Example \ref{ex:1}, which are similar to those for the scheme \eqref{scheme}.
This is consistent with the volume-locking-free
estimates in Theorem \ref{thm:2}.
However, the norm of the discrete solution for the scheme \eqref{scheme-2} for Example \ref{ex:2}
shows an upper bound depending on $h$ which 
indicates that it is not gradient-robust.
In the next subsection, we slightly modify the scheme \eqref{scheme-2} to
make it gradient-robust.
\begin{figure}
  \begin{center}
    \includegraphics[height=0.255\textwidth]{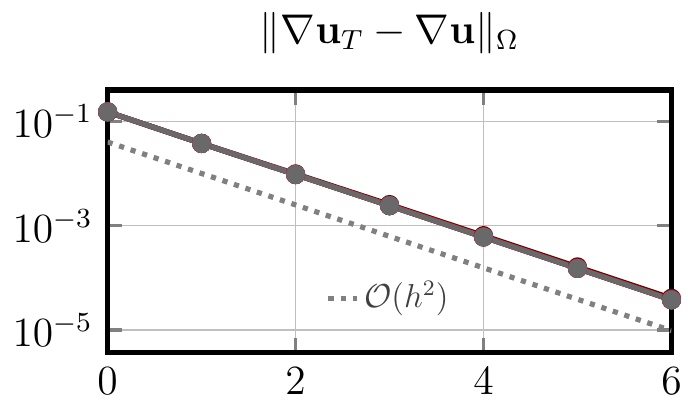}
    \hspace*{-0.02\textwidth}
\includegraphics[height=0.255\textwidth]{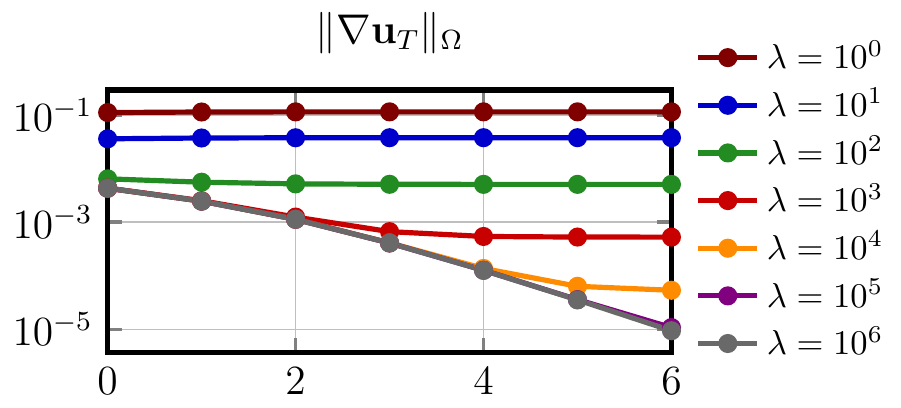}
\end{center}
  \caption{Discretizaton error for Example \ref{ex:1} (left) and norm of discrete error for Example \ref{ex:2} (right) for the method \eqref{scheme-2}, $k=2$, on a barycentric-refined mesh under mesh refinement ($x$-axis: refinement level $L$) and different values of $\lambda$ for Example \ref{ex:2}.} 
  \label{fig:H2ab}
  \hspace*{-0.02\textwidth}
\end{figure}


\subsection{Gradient-robust relaxed $H(\Div)$-conforming HDG scheme}
As in Section \ref{sec:hdivhdg:gradrobhdiv} we consider the equivalent DG formulation
\begin{equation}\tag{S2-DG}\label{eq:S2-DG}
  \hat a_h(\bld u_T,\bld v_T) = \hat{f}(\bld v_T) \quad \forall \bld v_T\in \Sh^-,
\end{equation}
If we consider a splitting as in \eqref{eq:def:decomp} with
\begin{subequations}\label{eq:def2:decomp}
\begin{equation}\label{eq:def2:Uh0}
  \Sh^{\,-,0} := \{ \bld v_T \in \Sh^{\,-} : \Div \bld v_T = 0, \ \forall T \in \Oh
  \}
\end{equation}
and 
\begin{equation}\label{eq:def2:Uhp}
  \Sh^{\,-,\perp} := \{ \bld v_T \in \Sh^{\,-} :
    \hat a_h(\bld v_T, \bld w_T) = 0, 
    \forall \bld w_T \in \Sh^{\,-,0} \},
  \end{equation}
\end{subequations}
we can again decompose every discrete function $\bld v_T \in \Sh^{\,-}$ as
$\bld v_T = \bld v_T^0 + \bld v_T^\perp$ with $\bld v_T^0 \in
\Sh^{\,-,0},\bld v_T^\perp \in \Sh^{\,-,\perp}$.

\begin{subequations} \label{eq:decomp2}
  \begin{align}
    \hat a_{h}(\bld u_T^0, \bld v_T^0) & = \hat f(\bld v_T^0) \quad \forall ~ \bld v_T^0 \in \Sh^{\,-,0}, \label{eq:auh0vh02}\\
    \hat a_{h}(\bld u_T^\perp, \bld v_T^\perp)
  & = \hat f(\bld v_T^\perp) \quad \forall ~ \bld v_T^\perp\in \Sh^{\,-,\perp}.
        \label{eq:auhpvhp2}
  \end{align}
\end{subequations}

Note that Theorem \ref{thm:gradientrobustnessh} does not directly translate to the relaxed $H(\Div)$-conforming case only because \eqref{eq:fvh0} does not hold as the facet normal jumps do not vanish. However, we can introduce a modification in the treatment of the 
right hand side that re-enables gradient-robustness.
The modified scheme is: Find $\compUh\in \compUhspace^{\,-}$ such that 
\begin{align}\tag{S3}
 \label{scheme-3}
 a_h(\compUh,\compVh) = f( (\Pi_V \bld v_T, 0)), \quad \forall \compVh\in \compUhspace^{\,-}.
\end{align}
or in the equivalent DG formulation: Find $\bld u_T \in \Sh^{\,-}$ such that 
\begin{align}\tag{S3-DG}
 \label{scheme-3-dg}
\hat a_h(\bld u_T,\bld v_T) = \hat f( \Pi_V \bld v_T), \quad \forall \bld v_T \in \Sh^{\,-}.
\end{align}
Here,
$\Pi_V$ is a generalization of the BDM interpolator, \cite[Proposition
2.3.2]{BoffiBrezziFortin13}, which can deal with only element-wise smooth
functions by averaging, cf. 
the appendix for a definition. 
\begin{remark}
  Let us note that the BDM interpolator is not mandatory here. In \cite{LLS_ARXIV_2017} and \cite{LLS_ARXIV_2018} several conditions on a suitable reconstruction operator are formulated. A much simpler version of the BDM interpolation operator is suggested that exploits the knowledge on the pre-image $\Sh^{\,-}$ and a proper basis for the relaxed $H(\Div)$-conforming finite element space. The reconstruction operation can then be realized by a simple averaging of a few unknowns which makes it computationally very cheap. In the numerical examples below we make use of this operator.
\end{remark}

\begin{lemma}
  The scheme \eqref{scheme-3-dg} is \emph{gradient-robust}, i.e. for $\bld f=\nabla \phi$, $\phi \in H^1(\Omega)$, the solution 
  $\bld u_T= \bld u_T^0  + \bld u_T^\perp
  \in \Sh^{\,-}$ has
  $
  \bld u_T^0 = \bld 0,~
  \bld u_T^\perp = \mathcal{O}(\lambda^{-1}).
  $
\end{lemma}
\begin{proof}
  With $\hat f(\cdot) = (\nabla \phi, \Pi_V \cdot)_{\Omega}$
  there holds after partial integration
\begin{equation}\label{eq:fvh1}
  \hat f(\bld v_T^0)= - \sum_{T \in \mathcal{T}_h} (\phi, \Div
  \Pi_V \bld v^0_T)_T + \sum_{F \in \mathcal{F}_h} ( \phi, \jmp{
  \Pi_V \bld v^0_T \cdot \bld n}_F ) = 0~ \forall ~ \bld v_T^0 \in
  \Sh^{\,-,0}.
\end{equation}
where we used $\Div \Pi_V \bld v^0_T=0$ cf. \cite[Lemma
4.8]{LLS_ARXIV_2017} and $\jmp{ \Pi_V \bld v^0_T \cdot \bld n}_F = 0$. The remainder of the proof follows from the proof of Theorem \ref{thm:gradientrobustnessh}.
\end{proof}
For the robustness of the scheme we give the following improved version of Lemma \ref{lemma:cs} (in the DG setting).
\begin{lemma}
 \label{lemma:cs2}
 Let $\bld u\in [H_0^2(\Omega)]^d$ be the solution to the equations \eqref{eqns} and define the splitting $\bld f = \bld f^\mu + \bld f^\lambda$ with $\bld f^\mu = - \Div\left(2\mu\gradss \bld u\right)$ and $\bld f^\lambda = -  \grads\left(\lambda\,\Div \bld u\right)$
 and $f(\cdot) = f^{\mu}(\cdot) + f^{\lambda}(\cdot)$ and $\hat f(\cdot) = \hat f^{\mu}(\cdot) + \hat f^{\lambda}(\cdot)$ correspondingly.
 Denote $\compU:=(\bld u, \bld u^t)\in \compUspace$.
 There holds for all $\compV=(\bld v_T,\bld v_F)\in \compUhspace^{\,-}$ 
 \begin{subequations}
 \begin{align}
  \label{csmu2}
  a_h^{\mu}(\compU, \compV) &= \hat f^\mu(\Pi_V \bld v_T) + \widetilde{\mathcal{E}}_c^\mu(\bld u, \compV),\\
  \label{cslambda2}
  a_h^\lambda(\compU, \compV) &= \hat f^\lambda(\Pi_V \bld v_T), \\ 
  \label{cs2}
  a_h(\compU, \compV) &= \hat f(\Pi_V \bld v_T) + \widetilde{\mathcal{E}}_c^{\mu}(\bld u, \compV), \\
\text{with } \qquad 
\widetilde{\mathcal{E}}_c^{\mu}(\bld u, \compV)
& = \mathcal{E}_c^{\mu}(\bld u, \compV) + \hat f^\mu(\bld v_T - \Pi_V \bld v_T). 
 \end{align}
\end{subequations}
Moreover, for $\bld u\in [H_0^\ell(\Omega)]^d$, $\ell\ge 2$ and $1\le m\le\min(k,\ell-1)$ we have
\begin{align}
  \widetilde{\mathcal{E}}_c^{\mu}(\bld u, \compV)
& \preceq 
h^m \mu^{1/2}\|\bld u\|_{m+1}\|\compV\|_{\mu,h}.
  \label{cs-term-est2}
\end{align}
\end{lemma}
\begin{proof}
  From \eqref{csmu} the result \eqref{csmu2} follows directly.
  Next, we note that
 $\Div \Pi_V \bld v_T = \Div \bld v_T$ for $\bld v_T \in \Sh^-$. This, we can see from the following observation. Let $q \in \mathbb{P}^{k-1}(T)$ and $T \in \mathcal{T}_h$. Then, we have
    \begin{align*}
      \int_{T} \Div(\Pi_V \bld v_T ) q \, dx 
      & = - \int_{T} \Pi_V \bld v_T  \cdot \nabla q \, dx +  \int_{\partial T} \Pi_V \bld v_T \cdot n \, q \, ds \\
      & = - \int_{T} \bld v_T \cdot \nabla q \, dx + \int_{\partial T} \bld v_T \cdot n \, q \, ds
      = \int_{T} \Div(\bld v_T) q \, dx 
    \end{align*}
where we exploited \eqref{eq:ass2B} and \eqref{eq:ass2} of the BDM interpolation. As $\Div(\bld v_T), \Div(\Pi_V \bld v_T) \in \mathbb{P}^{k-1}(T)$ we obtain $\Div(\bld v_T) = \Div(\Pi_V \bld v_T)$ pointwise. Then, \eqref{cslambda2} follows from partial integration:
  \begin{align*}
\hat    f^{\lambda}(\Pi_V \bld v_T) & = \sum_{T\in \mathcal{T}_h} \int_T - \nabla ( \lambda \Div \bld u ) \Pi_V \bld v_T \dx \\
                          & = \sum_{T\in \mathcal{T}_h} \int_T \lambda \Div \bld u \underbrace{\Div (\Pi_V \bld v_T)}_{= \Div \bld v_T} \dx  - \int_{\partial T} \lambda \Div \bld u \Pi_V \bld v_T \cdot \bld n \ds \\
                          & = a_h^{\lambda}(\compU,\compV) - \sum_{F \in \mathcal{F}_h \setminus \partial \Omega} \int_{F} \lambda \Div \bld u  \underbrace{\jmp{\Pi_V \bld v_T}_\ast}_{=0} \cdot \bld n \ds = a_h^{\lambda}(\compU,\compV). 
  \end{align*}
Next, we note that for $T \in \mathcal{T}_h$ there holds with standard Bramble-Hilbert arguments ($\bld v_T \in H^1(T)$)
\begin{equation}\label{eq:bhr}
  \| (\id - \Pi_V) \bld v_T \|_T^2 \preceq h \Vert \nabla \bld v_T \Vert_T
\end{equation}
as constants are in the kernel of $\id - \Pi_V$. Let further $\mathcal{P}^{m-2} \bld f$ be the element-wise $L^2$ projection into $[\Pi^{m-2}(T)]^d,~T\in\mathcal{T}_h$. Then, we have
  \begin{align*}
    ( \bld f^{\mu},& \bld v_T - \Pi_V \bld v_T)
    =    ( \bld f^{\mu} - \mathcal{P}^{m-2} \bld f^{\mu}, \bld v_T - \Pi_V \bld v_T) 
    \leq    \Vert \bld f^{\mu} - \mathcal{P}^{m-2} \bld f^{\mu} \Vert \Vert \bld v_T - \Pi_V \bld v_T \Vert \\
    &\preceq  h^{m-1}  \Vert \bld f^{\mu} \Vert_{m-1}~  h \Vert \bld v_T \Vert_{1,h} 
    \preceq  h^{m}  \mu \Vert \bld u \Vert_{m+1}~  \Vert \compV \Vert_{1,h} 
    \preceq  h^{m}  \mu^{\frac12} \Vert \bld u \Vert_{m+1}~  \Vert \compV \Vert_{\mu,h}.
  \end{align*}
Here, we made use of \eqref{eq:ass2} in the last step.
\end{proof}

Finally, the locking-free error estimates for the scheme \eqref{scheme-3}
is given below.
\begin{theorem}
\label{thm:3}
 Assume $k\ge 2$ and the regularity $\bld u\in [H^{k+1}(\Omega)]^d$.
 Let
 $\bld u_T \in \Sh^{\,-}$ be the numerical solution to the scheme \eqref{scheme-3-dg} (or equivalently 
 $\compUh = (\bld u_T, \mathcal{L}_h(\bld u_T)) \in \compUhspace^{\,-}$ the numerical solution to \eqref{scheme-3}).
 Then, for sufficiently large stabilization parameter $\alpha_0$, the estimates \eqref{est-1}--\eqref{est-2} hold.
\end{theorem}
\begin{proof}
Proceeding as in the proof of Theorem \ref{thm:2} (and hence using the equivalent HDG-version again) with $\compVh = (\Pi_V \bld u, \Pi_M \bld u)\in \compUhspace\subset \compUhspace^{\,-}$ and $\compWh := \compUh - \compVh \in \compUhspace^{\,-}$, we obtain
\begin{align*}
&  \|\compWh\|_{\mu,h}^2
  +\lambda \|\Div(\bld w_T)\|^2\\
& \preceq a_h(\compWh, \compWh) 
 = a_h(\compU-\compVh, \compWh) - \widetilde{\mathcal{E}}_c^{\mu}(\bld u,\compWh)\\
 &= a_h^{\mu}(\compU-\compVh,\; \compWh) 
	+ \underbrace{a_h^{\lambda}(\compU-\compVh,\; \compWh)}_{=0}  -\widetilde{\mathcal{E}}_c^{\mu}(\bld u,\compWh) 
 \\
&\preceq \left( \|\compU-\compVh\|_{\mu,\ast,h} 
+\mu^{\frac12}h^k\|\bld u\|_{k+1}\right) \|\compWh\|_{\mu,h}.
\end{align*}
With interpolation estimates for $\|\compU-\compVh\|_{\mu,\ast,h} $ this implies
\begin{align*}
\|\compUh-\compVh\|_{\mu,h}
  +\lambda^{\frac12}\|\Div(\bld u_T-\bld v_T)\|
  \preceq   \mu^{\frac12} h^k \|\bld u\|_{k+1}.
\end{align*}
Then, the estimates \eqref{est-1} and  \eqref{est-1b} follow from triangle inequalities.

For the $L^2$-estimate, let $\bld \phi$ be the solution to the dual problem \eqref{dual-eq} with 
$\bld \theta = \Pi_V (\bld u- \bld u_T)$ and $\compSh \in \compUhspace$ the corresponding interpolation as before.
Noting that
	$\widetilde{\mathcal{E}}_c^{\mu}(\cdot,\compWh)$ does not depend on $\bld w_F = \bld u_F-\Pi_M \bld u $, cf. Lemma \ref{lemma:cs} and Lemma \ref{lemma:cs2},
and $\compS = (\bld \phi, \bld \phi^t)$ 
we get for
$\Pi \compV = \Pi (\bld v_T, \bld v_F) = (\Pi_V \bld v_T, \Pi_M \bld v_F)$, $\compV = (\bld v_T, \bld v_F) \in \compUspace$
\begin{align*}
  \|\Pi_V (\bld u- \bld u_T) \|_\Omega^2 &=
a_h(\compS, \Pi(\compU - \compUh))  - \overbrace{\widetilde{\mathcal{E}}_c^{\mu}(\bld \phi, \Pi(\compU - \compUh))}^{=0}\\
&= 
a_h(\compS, \compU-\compUh) 
- a_h(\compS, (\id-\Pi)(\compU-\compUh)) 
\\[-3ex]
&=
a_h(\compS - \compSh, \compU-\compUh) 
- a_h(\compS, (\id-\Pi)(\compU-\compUh)) 
+
\overbrace{\widetilde{\mathcal{E}}^{\mu}_c(\bld u, \compSh)}^{=0}
\\
&=
a_h(\compS - \compSh, \compU-\compUh) 
- \underbrace{(\bld \theta,(\id-\Pi_V)(\bld u-
\bld u_T))}_{( \Pi_V(\bld u-
\bld u_T),(\id-\Pi_V)(\bld u-
\bld u_T)) = 0}
\\
& \preceq 
h ( \mu \Vert \bld \phi \Vert_2 + \lambda \Vert \Div \bld \phi \Vert_1) (\mu^{-\tfrac12} \Vert \compU - \compUh \Vert_{\mu,\ast,h} + \Vert (\id - \Pi_Q) \Div \bld u_T \Vert) \\
 & \preceq
 \| \Pi_V (\bld u- \bld u_T) \|_\Omega
   \cdot \left(
   h \left( \mu^{-\tfrac12} \Vert \compU - \compUh \Vert_{\mu,\ast,h} + \Vert (\id - \Pi_Q) \Div \bld u \Vert \right) \right)
\end{align*}
Dividing by $\| \Pi_V (\bld u- \bld u_T) \|_\Omega$ and applying the triangle inequality:
$$
\| \bld u- \bld u_T \|_\Omega \leq \| \Pi_V (\bld u- \bld u_T) \|_\Omega +
\underbrace{\| (\id - \Pi_V) (\bld u- \bld u_T) \|_\Omega}_{\preceq h \Vert \compU - \compUh \Vert_{1,h}}
$$
yields
$$
\| \bld u- \bld u_T \|_\Omega \preceq  h \left( \mu^{-\tfrac12} \Vert \compU - \compUh \Vert_{\mu,\ast,h} + \Vert (\id - \Pi_Q) \Div \bld u \Vert \right)
$$
and hence the claim.
\end{proof}

With this result we conclude that method \eqref{scheme-3} has quasi-optimal a-priori error bounds and is free of locking, i.e. it is volume-locking free and gradient-robust.

\subsection{Numerical results for 
the scheme \eqref{scheme-3}}
The numerical results for the 
two examples in Section \ref{sec:mot} 
for the scheme \eqref{scheme-3} are given in Figure
\ref{fig:H3ab}. 
The results are essentially similar to those for 
the scheme \eqref{scheme}. In particular, 
we observe that the discrete norms in Example \ref{ex:2}
are essentially independent of $h$.
\begin{figure}
  \begin{center}
    \includegraphics[height=0.255\textwidth]{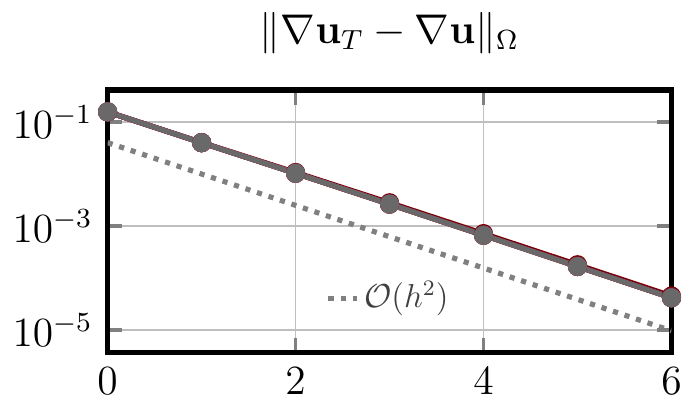}
    \hspace*{-0.02\textwidth}
\includegraphics[height=0.255\textwidth]{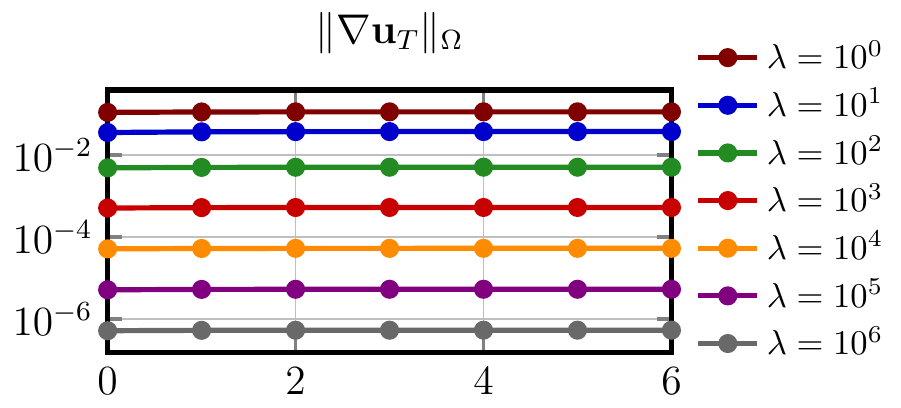}
    \hspace*{-0.06\textwidth}
  \end{center}    
  \caption{Discretizaton error for Example \ref{ex:1} (left) and norm of discrete error for Example \ref{ex:2} (right) for the method \eqref{scheme-3}, $k=2$, on a barycentric-refined mesh under mesh refinement ($x$-axis: refinement level $L$) and different values of $\lambda$ for Example \ref{ex:2}.} 
  \label{fig:H3ab}
  \hspace*{-0.02\textwidth}
\end{figure}


\section{Conclusion}
\label{sec:conclude}
The concept of gradient-robustness for numerical methods for linear
elasticity is introduced in this paper.
The class of divergence-conforming HDG methods are presented and analyzed as
an example of {volume-locking-free} and {gradient-robust} finite element
methods for linear elasticity. Two efficient variants of the base
divergece-conforming HDG scheme with reduced globally coupled degrees of
freedom are also discussed and analyzed.

\section*{Appendix. The BDM interpolator for discontinuous functions}\label{sec:bdm}
The BDM interpolator for discontinuous functions is defined element-by-element for $\bld v_T \in H^1(T)$ through
  \begin{subequations} \label{eq:ass}
    \begin{align}
      (\Pi_V \bld v_T \! \cdot \! \bld n, \varphi)_{F}
      & =
        (
       \{\!\!\{ \bld v_T \! \cdot \! \bld n \}\!\!\}_*, \varphi)_{F}
      && \forall~ \varphi \!\in\! \mathcal{P}^{k}\!(F), F \!\in\! \partial T,  \label{eq:ass2B}\\
      (\Pi_V \bld v_T, \varphi)_{T}
      & =
      (\bld v_T, \varphi)_{T}
      && \forall~ \varphi \in [\mathcal{N}^{k-2}(T)]^d,  \label{eq:ass2}
    \end{align}
  \end{subequations}
with $\mathcal{N}^{k-2}:=[\mathcal{P}^{k-2}(T)]^d + [\mathcal{P}^{k-2}(T)]^d \times x$ and $\{\!\!\{ \cdot \}\!\!\}_*$ the usual DG average operator, cf. \cite{CockburnKanschatSchotzau05,guzman2016h}.

\bibliographystyle{siam}
\bibliography{all}

\end{document}